\documentclass[preprint]{elsarticle}

\usepackage{natbib}
 \bibpunct[, ]{(}{)}{,}{a}{}{,}%
 \def\newblock{\ }%
\usepackage{amsthm}
\usepackage{amssymb}
\usepackage{dsfont}
\usepackage{bm}
\usepackage{enumitem}

\usepackage[utf8]{inputenc}
\usepackage[english]{babel}
\usepackage{amsmath}
\usepackage{graphicx}
\usepackage[toc,page]{appendix}
\usepackage{amssymb}
\usepackage{enumitem}    
\usepackage{rotating}
\usepackage{amssymb}
\usepackage[table]{xcolor}

\usepackage[toc,page]{appendix}

\usepackage[utf8]{inputenc}
\usepackage{caption}
\usepackage{subcaption}
\usepackage{textcomp}
\usepackage{algorithm}
\usepackage{algpseudocode}
\usepackage[a4paper, total={6in, 10in}]{geometry}
\usepackage{bm}
\usepackage{chngcntr}
\usepackage{apptools}
\AtAppendix{\counterwithin{lem}{section}}
\usepackage{booktabs}
\usepackage{tikz}
\usepackage{pgfplots}
\usetikzlibrary{arrows,automata,fit,backgrounds,scopes,shadows,shapes}
\usetikzlibrary{arrows.meta, arrows.spaced, positioning}
\usetikzlibrary{decorations.pathreplacing,angles,quotes}
\usetikzlibrary{calc,through,intersections}
\usepgfplotslibrary{fillbetween}
\tikzset{
	node distance=1.5cm,
	every node/.style={font=\small},
	every state/.style={fill=white,draw=black,thick,
		text=black,scale=1,inner sep=.3mm,minimum size=3.5mm},
	abs/.style={state, fill=gray},
	tr/.style={->,>=latex, thick,shorten >=0.5pt},
	slb/.style={tr, loop below, out=300, in=240,min distance=0.75cm},
	sla/.style={tr, loop above, out=60, in=120, min distance=0.75cm}
}

\usepackage{url}
\usepackage{etoolbox}
\AtBeginEnvironment{tabular}{\tiny}
\newtheorem{proposition}{Proposition}
\newtheorem{corollary}{Corollary}

\begin{document}





\title{An Unconstrained Optimization Approach to Moment Fitting with Phase Type Distributions}

\author[inst1]{Eliran Sherzer}
\author[inst2]{Yehezkel S. Resheff}
\author[inst3]{Miklos Telek}

\address[inst1]{Industrial Engineering, Ariel University.}
\address[inst2]{Business school of the Hebrew University.}
\address[inst3]{Department of Networked Systems and Services,  Budapest University of Technology and Economics.}

\begin{abstract}
    Phase type (PH) distributions are widely used in modeling and simulation due to their generality and analytical properties. In such settings, it is often necessary to construct a PH distribution that aligns with real-world data by matching a set of prescribed moments. Existing approaches provide either exact closed-form solutions or iterative procedures that may yield exact or approximate results. However, these methods are limited to matching a small number of moments using PH distributions with a small number of phases, or are restricted to narrow subclasses within the PH family. We address the problem of approximately fitting a larger set of given moments using potentially large PH distributions. We introduce an optimization methodology that relies on a re-parametrization of the Markovian representation, formulated in a space that enables unconstrained optimization of the moment-matching objective. This reformulation allows us to scale to significantly larger PH distributions and capture higher moments. Results on a large and diverse set of moment targets show that the proposed method is, in the vast majority of cases, capable of fitting as many as $20$ moments to PH distributions with as many as $100$ phases, with small relative errors on the order of under $0.5\%$ from each target. We further demonstrate an application of the optimization framework where we search for a PH distribution that conforms not only to a given set of moments but also to a given shape. Finally, we illustrate the practical utility of this approach through a queueing application, presenting a case study that examines the influence of the $i^{\text{th}}$ moment of the inter-arrival and service time distributions on the steady-state probabilities of the $GI/GI/1$ queue length.
\end{abstract}



%
\maketitle

\section{Introduction}


Stochastic models of real-life systems are hard to solve in general because real system behavior often results in complicated stochastic models that are both analytically and numerically intractable. One of the few exceptions is when the behavior of a system can be accurately modeled by a continuous-time Markov chain (CTMC). Unfortunately, this structure poses severe modeling constraints for the system's behavior. Namely, all stochastic activities in the system have to be independent and exponentially distributed, which is hardly the case in most real systems. Phase type (PH) distributions \cite{phbook} offer a way to relax this modeling limitation by describing the non-exponentially distributed stochastic activity with the help of a small Markov chain. The resulting stochastic model remains a CTMC, for which efficient numerical procedures are available. The key to applying this analysis approach is the approximation of non-exponential distributions with PH distributions. 

A vast literature deals with precisely how to obtain the PH parameters, where depending on the available information about the non-exponential distribution, the aim is to find a proxy distribution that provides an approximation either for the underlying distribution itself, or a sequence of its moments. In this work, we focus on the second case, but we provide a more general framework that also applies to distribution approximation. 

A justification for preferring moment matching over distribution approximation can be found, for instance, in queueing systems. In~\cite{sherzer23}, the authors use a deep learning approach and predict the stationary queue length of a $GI/GI/1$ queue. Notably, the input for the prediction relied solely on the first five moments of the inter-arrival and service time distributions. It is also demonstrated empirically that when the inter-arrival and service time distributions are not provided directly but instead through realizations sampled from those distributions, fitting moments proves significantly more effective than fitting the entire distribution for obtaining the steady-state probabilities of a $GI/GI/1$ queue. An additional example can be found in~\cite{SHERZER2024}, where the first inter-arrival and service time moments are used for predicting the transient queue length distribution of a $G(t)/GI/1$ queue.

PH distributions are helpful tools for reasoning about many stochastic models. For example, the $GI/GI/1$ model can be analyzed from data in the following way: (i) extract the first five or more moments of the inter-arrival and service time distributions. (ii) Fit a PH distribution for arrivals and service with respect to the moments from (i). Finally, (iii) the steady-state queue length distribution is derived via the analysis of the underlying Quasi Birth and Death (QBD) process (see Section 21.4 of~\cite{Harchol2013}), which can solve the resulting $PH/PH/1$ queue.  Obtaining a PH distribution from moments can be helpful also in more complex queueing systems~\cite{He2012} and other stochastic models, such as inventory problems~\cite{Manuel2007}. Even in stochastic models lacking analytical or numerical solutions, fitting a PH distribution to moments can be beneficial. This is due to the fact that PH distributions are straightforward to sample from, making simulation-based analysis a practical alternative.

Another advantage of using moments for fitting PH distributions is that moments are readily obtainable. In empirical settings, they can be directly estimated from data, while in analytical models, they are typically derivable from the underlying random variables. Moreover, in stochastic systems—such as inter-departure times in queues~\cite{Stanford1989}—the Laplace–Stieltjes Transform (LST) can often be derived, enabling straightforward computation of moments. This accessibility broadens the applicability of moment-based fitting models across a wide range of scenarios.

Fitting moments of PH distributions can also be used as a tool to study properties of stochastic models. For example, it enables the exploration of the impact of the $k^{\text{th}}$ moment on a stochastic model. This type of analysis is valuable both for gaining deeper insight into the stochastic model, and from a practical perspective for determining how many moments are necessary for accurate fitting under various conditions (e.g., queueing network structure, utilization levels, etc.). For the sake of demonstration, in Section~\ref{sec:queue}, we revisit the $PH/PH/1$ queueing model and compute the queue length distribution using fitted PH distributions based on the first 2, 3, 4, and 5 moments or inter-arrival and service times. We then compare these results to the distribution obtained from the original ground-truth model. This comparison allows us to assess the influence of each additional moment on system behavior under various conditions. While a comprehensive investigation of this effect is beyond the scope of this paper, we highlight the potential of our approach as a tool for such analyses.

The existing body of work on moment-based approximation has primarily focused on moment matching, where a predefined set of target moments—typically those of a non-exponential distribution—is exactly matched by a phase-type (PH) distribution. However, this exact matching approach limits scalability, as it generally supports matching fewer than 10 moments and often imposes constraints on the state-space of the resulting PH distribution~\cite{Johnson1991,Johnson01011989, HORVATH2009396}. By relaxing the requirement for exact matching, it becomes possible to fit a greater number of moments and to utilize any PH structure. This is naturally achieved by formulating the fitting task as a constrained optimization problem that seeks a PH representation consistent with a specified moment set. A heuristic approach in this spirit for the family of acyclic PH distributions was previously proposed in~\cite{buchholz2009heuristic}.

The key novelty in the current work lies in a reparameterization of the PH distribution, which transforms the constrained optimization problem into an unconstrained one. This enables us to accurately fit PH distributions with hundreds of phases to as many as 20 moments with deviations of less than one percent from the each of the target values. This method supports PH representations on the order of hundreds of states while maintaining practical runtimes—typically within minutes, and in the most complex cases tested, never exceeding a few hours.

The proposed method is not restricted to fitting moments. The optimization framework proposed is suitable also for fitting other statistical measures that are specified as differential functions of the parameters of the PH distribution, such as specific values of the PDF and the CDF. Furthermore, it can be applied simultaneously to fit moments and shape requirements. For example, it would be possible to fit the first eight moments, the PDF at some point $x_0$, and the CDF in the $50^{th}$ percentile. This is demonstrated in Section~\ref{sec:shape}.

The main contribution of this paper is as follows. First, we formulate a reparameterization of the PH distribution allowing unconstrained optimization of differential functions of the original Markovian parameters. We further suggest specific lower-dimensional reparameterizations in the same spirit for the special cases of Coxian and Hyper-Erlang distributions. Second, we apply the method to moment fitting and demonstrate the accuracy of the method in a large set of numerical experiments. We also demonstrate the application of the method for the task of fitting a PH with a given set of moments and a given shape. To the best of our knowledge, this is the first method to achieve this. An implementation of our approach can be found in the GitHub repository:  \url{https://github.com/Hezi-Resheff/ph-fit}. Finally, we propose and demonstrate an application of the moment fitting ability to the study of the effect of higher moment accuracy on queueing measures. 



The rest of the paper is organized as follows. In Section~\ref{sec:liter}, we present a literature review.  In Section~\ref{sec:PH-dist-about}, we shortly describe PH distribution properties. In Section~\ref{sec:opt_approach}, we present our optimization approach. In Section~\ref{sec:framework}, we delve into more technical details and describe how our approach is implemented. In Section~\ref{sec:experiment}, we describe experiments to test our method, and in Section~\ref{sec:results} we present the results. In Section~\ref{sec:shape} we demonstrate how the model can fit moments and shape simultaneously. In Section~\ref{sec:queue} we present a queueing application where we study the impact of the inter-arrival and service time $l^{th}$ moment on the queue length distribution. Finally, we conclude the paper in Section~\ref{sec:Conclusions}.

\section{Literature}\label{sec:liter}

The literature is rich with studies focused on obtaining a Markovian PH representation for a given set of moments. Notably, these studies mostly aim to obtain an exact solution (i.e., moment matching), or are limited to specific structures and small sizes. In contrast, our approach provides an approximation that is applicable to general PH distributions and scales to large sizes both in the number of phases and in the number of moments. In this section, we position our work within the existing literature and highlight that, while our method does not yield an exact solution, it introduces significant advancements in other key aspects.

Moment fitting and matching algorithms can be assessed based on four key criteria~\cite{OSOGAMI2006524}: (i) \textbf{Number of Moments Matched}: Generally, matching a larger number of moments is more favorable.
(ii) \textbf{Computational Efficiency}: Algorithms should be efficient, ideally with short execution times. (iii) \textbf{Generality of the Solution}: The algorithm should ideally be applicable to a broad range of distributions. (iv)
\textbf{Minimal Number of Phases}: The resulting matched PH distribution should have as few phases as possible while matching the target moments, to maintain simplicity and efficiency. It is critical at this point to distinguish between the positive property of finding a small PH that matches the desired target moments, from the negative property of being \textit{limited} to small PHs due to theoretical or computational limitations of a method. 

Numerous moment-matching algorithms have been proposed in prior research. While these methods excel in specific areas, they often fall short in at least one of the criteria mentioned above.

When matching only the first two moments is sufficient, solutions that perform well across criteria (ii), (iii) and (iv) can often be achieved. For example, Sauer and Chandy~\cite{5391223} provide a closed-form solution for matching the first two moments of a general distribution with a positive squared coefficient of variation (SCV), within the $PH(3)$ domain.  They employ a two-phase hyper-exponential distribution for $SCV > 1$ and a generalized Erlang distribution (which is an Erlang where the rates do not have to be equal)  for $SCV < 1$. Similarly,~\cite{10.1145/1009375.806155} presents a closed-form solution for matching the first two moments of a general distribution for  $SCV > 0$. Specifically, a two-phase Coxian$^+$ PH distribution is applied for $SCV>1$, and a generalized Erlang for distributions with $SCV <1$, where Coxian$^+$ is a Coxian distribution with no mass probability at zero.     

When the focus is restricted to matching a subset of distributions, solutions that excel in computational efficiency and minimality can be achieved. \cite{doi:10.1287/opre.30.1.125} and~\cite{Altiok01061985} focus on distributions with $SCV > 1$ and sufficiently high third moments. They derive closed-form solutions for matching the first three moments, with Whitt utilizing a two-phase hyper-exponential distribution and Altiok employing a two-phase Coxian$^+$  PH distribution.

Further contributions include methods that match three moments~\cite{Bobbio01012005, OSOGAMI2006524}.~\cite{Bobbio01012005} presents a technique for constructing minimal-order acyclic phase type (APH) distributions given the first three moments. \cite{OSOGAMI2006524} narrows the search space to Erlang–Coxian (EC) distributions, significantly improving computational efficiency.

\cite{Johnson1991,Johnson01011989} address the first three moments while emphasizing generality and minimality. They propose a closed-form solution for matching the first three moments of any distribution $G \in \text{PH}(3)$, using a Hyper-Erlang distribution with a common order, though their method requires $2~\text{OPT}(G) + 2$ phases in the worst case, where OPT(G) is defined as the minimum number of phases required in an acyclic PH distribution, which permits a probability mass at zero, to accurately represent the first three moments of a given distribution $G$. In subsequent work~\cite{Johnson01011990,Johnson01011990_a}, they extend their approach by utilizing three types of PH distributions: Hyper-Erlang distributions, a Coxian$^+$ PH distribution, and a general PH distribution. Their nearly minimal solution requires at most $\text{OPT}(G) + 2$ phases, though it involves solving computationally expensive nonlinear programming problems. In our fitting model, we rely on the three PH structures, namely, Hyper-Erlang, Coxian and a general PH.

For matching more than three moments,~\cite{HORVATH2009396} provide a method to find a  PH(3) distribution that matches the first five moments.  A notable advantage is its ability to produce minimal PH sizes, optimizing the phase count criterion. In contrast, our method is not restricted to the $PH(3)$ domain. 

For higher-order matching,~\cite{casale_et_al:DagSemProc.07461.8} construct hyper-exponential distributions of size $n$ based on $2n - 1$ moments. While this method fits more than five moments, its application is limited to a small subset of PH distributions, reducing generality. \cite{Horváth08052007} provide a numerical procedure for matching $2n - 1$ moments to an acyclic PH (APH) of size $n$, though it becomes impractical for $n > 8$. In another study~\cite{TELEK20071153}, the same authors propose a minimal-size PH distribution method that addresses all four criteria but struggles near the Markovian process boundary due to increased iteration requirements.

\cite{10.1007/978-3-642-39408-9_17} introduces generalized Hyper-Erlang distributions, obtaining all PH distributions matching target moments with a Markovian representation up to a specified size. While versatile, this approach is constrained by the computational difficulty of solving polynomial systems, limiting the number of moments it can match. \cite{buchholz2009heuristic,buchholz2014input} formulates an optimization-based approach for moment fitting for acyclic phase-type distributions. Their method is an alternating least squares approach, alternating between small constrained optimization problems for different partitions of the parameters until convergence. In our experiments, this method was successful only for low order acyclic PH distributions and for fitting a small number of moments (See Appendix), although we believe failure for larger sizes is probably due to numeric instability rather than an inherent limitation of the method. 

In summary, regarding criteria (i), (ii), and (iii), no existing method can achieve high-order moment matching for a wide range of distributions within a practical timeframe. By allowing for an approximate match, our method successfully addresses this challenge. However, with respect to criterion (iv), our method does not ensure a minimal number of phases, which we leave as an avenue for future research.

\section{PH distributions}
\label{sec:PH-dist-about}

PH distributions are a versatile family of probability distributions that model the time until absorption in a finite-state CTMC. They are defined over the non-negative real numbers and can approximate any non-negative distribution arbitrarily closely, making them highly flexible for modeling. A PH distribution is characterized by the pair \((\bm{\alpha}, \bm{T})\), where \( \bm{\alpha} \) is the initial probability vector, and \( \bm{T} \) is the subgenerator matrix defining transitions between transient states. These elements must satisfy specific constraints to ensure a valid PH distribution. A concise summary of the notation used in this paper can be found in the appendix (\ref{sec:appC}).

\textbf{Subgenerator Matrix}:
The matrix $\bm{T}  \in \mathbb{R}^{n \times n}$ describes the transitions between \( n \) transient states in the Markov chain. It must satisfy the following constraints: (a) Square Matrix:  
    \( \bm{T} \) is an \( n \times n \) square matrix, where \( n \) is the number of transient states.
(b) Off-Diagonal Elements (Transition Rates):  
    The off-diagonal elements \( T_{ij} \) must satisfy:
    \[
    T_{ij} \geq 0, \quad \text{for } i \neq j,
    \]
    representing the transition rates from state \( i \) to state \( j \).
 (c) Row Sums:  
    The row sums of \( \bm{T} \) must be non-positive:
    \[
    \sum_{j=1}^n T_{ij} \leq 0, \quad \forall i.
    \] This also implies that $T_{ii} < 0$, $\forall i$. (d) The subgenerator matrix must be non-singular. 
    
\textbf{Initial Probability Vector}:
The vector $\bm{\alpha}  \in \mathbb{R}^{1 \times n}$  specifies the starting probabilities in each transient state. It must be a probability vector satisfying: $  \alpha_i \geq 0$ and $    \sum_{i=1}^n \alpha_i = 1$. In total, a Markovian representation of a  PH distribution is defined as the pair $(\bm{\alpha}, \bm{T})$ as follows: 
\begin{equation}
\bm{\alpha}^T = 
\begin{bmatrix} \alpha_1 \\ \alpha_2 \\ \alpha_3 \\ \vdots \\ \alpha_n  \end{bmatrix},
\quad
\bm{T} = 
\begin{bmatrix}
T_{11} & T_{12} &  \cdots & T_{1n} \\
T_{21} & T_{22} &  \cdots & T_{2n} \\
\vdots & \vdots & \ddots & \vdots \\
T_{n1} & T_{n2} &  \cdots & T_{nn} \\
\end{bmatrix} 
\label{eq:parmetrization-1}
\end{equation}

\noindent where $\forall i: \alpha_i \geq 0$, $\forall i, j\neq i: T_{ij} \geq 0$, $\sum_i \alpha_i = 1$, $\forall i: \sum_j T_{ij} \leq 0$, and $\bm{T}$ is non-singular.

We note that whenever the row sums are strictly negative, $T$ is a strictly diagonally dominant matrix and is guaranteed to be non-singular by the Gershgorin Circle Theorem, leaving this constraint active only for the boundary cases that include a row sum of exactly zero. In practice, however, the non-singularity constraint does not appear to be a limiting factor in the numerical approaches and thus is omitted for the sake of simplicity in the rest of the paper. 

The $i$-th moment of a PH distribution with representation $(\bm{\alpha}, \bm{T})$ is computed as: 

$$ m_i = i! (-1)^i \bm{\alpha} \bm{T}^{-i} \mathds{1}_n $$

where $\mathds{1}_n$ is the length $n$ vector of ones. This property lends itself to moment fitting via direct optimization. In the next section, we formalize this idea. 

\section{An optimization approach to moment matching}\label{sec:opt_approach}
Consider the general $l$-moment matching objective for a PH distribution. That is, finding PH parameters $\bm{\alpha}, \bm{T}$, such that:

\begin{equation}
\label{eq:mom}
    \forall i\leq l: i! (-1)^i \bm{\alpha} \bm{T}^{-i} \mathds{1}_n \approx m_i, 
\end{equation}
\noindent here $m_i$ is the target $i-th$ moment. The best possible approximation for a PH($n$) of some chosen order $n$  could be obtained, in principle, by solving the weighted regression optimization problem:
\begin{align} 
\label{eq:opt-naive}
&\min_{ \bm{\alpha},\bm{T}}  \quad  \sum_{i=1}^{l} w_i \left( i! (-1)^i \bm{\alpha} \bm{T}^{-i} \mathds{1}_n - m_i \right)^2 \\ \nonumber
&\text{subject to:~} 
\forall i: \alpha_i \geq 0, \forall i,j\neq i: T_{ij} \geq 0, \sum_i \alpha_i = 1, \forall i: \sum_j T_{ij} \leq 0, det(\bm{T}) \ne 0 
\end{align}

\noindent where the weights $w_i$ for each of the terms allow to trade-off accuracy and account for the often expanding scale of the sequence of moments, for example by using a weight inversely proportional to the scale of each moment; setting $w_i = m_i^{-2}$ sums the square of the relative error of the moments. However, directly optimizing this constrained matrix-polynomial objective is hard. The core of the proposed method is designing constraint-free re-parametrizations for the PH distribution and some specific subtypes of it. These new parametrizations lead to simple equivalent optimization problems that are solvable by simple optimization methods such as gradient descent (GD). 

\subsection{The general PH re-parametrization}\label{sec:general_representation}

We begin with a general re-parametrization for the family of PH distributions of size $n$. The main idea is represent the PH in an (almost) unconstrained parameter space, with a differentiable mapping onto the properly constrained (Markovian) form of PH (see Section \ref{sec:PH-dist-about}). Such a form allows unconstrained optimization of any differentiable function of the original PH parameters with GD methods. The proposed general parametrization is as follows:
\begin{equation}
\bm{a}^T = 
\begin{bmatrix} a_1 \\ a_2 \\ a_3 \\ \vdots \\ a_n \\ \end{bmatrix},\quad
\bm{\gamma}^T = 
\begin{bmatrix} \gamma_1 \\ \gamma_2 \\ \gamma_3 \\ \vdots \\ \gamma_n \\ \end{bmatrix},\quad
\bm{Z} = 
\begin{bmatrix}
Z_{11} & Z_{12} & Z_{13} & \cdots & Z_{1n} \\
Z_{21} & Z_{22} & Z_{23} & \cdots & Z_{2n} \\
Z_{31} & Z_{32} & Z_{33} & \cdots & Z_{3n} \\
\vdots & \vdots & \vdots & \ddots & \vdots \\
Z_{n1} & Z_{n2} & Z_{n3} & \cdots &Z_{nn} \\
\end{bmatrix} ,
\label{eq:general-re-parmetrization}
\end{equation}
\noindent with $\bm{a}^T \in \mathbb{R}^n$, $\bm{\gamma} \in (\mathbb{R} \setminus \{0\})^n$, $\bm{Z} \in \mathbb{R}^{n^2}$.

\begin{proposition}
The image of $\; \mathbb{R}^n \times (\mathbb{R} \setminus \{0\})^n \times \mathbb{R}^{n^2}$ under the following differentiable transformation $\pi$ is the set of valid PH distributions of size $n$ in the Markovian representation  (See Eq. \eqref{eq:parmetrization-1}), with no zero elements. In other words, the general re-parametrization (Eq. \eqref{eq:general-re-parmetrization}) maps onto the interior of the standard Markovian form of PH via the following differentiable mapping~$\pi$:
\begin{equation}
\begin{aligned} 
     &\pi_{\alpha}(\bm{a}, \bm{\gamma}, \bm{Z}) = \mathit{softmax}(\bm{a}) , \\
     & \pi_{T}(\bm{a}, \bm{\gamma}, \bm{Z}) = \text{diag}(\bm{\gamma}^2) \cdot \left[ \mathit{softmax}(\bm{Z}) - \left( \bm{I}_n + \mathit{softmax}(\bm{Z}) \circ \bm{I}_n \right) \right],
\end{aligned}
\label{eq:parametrization-map}
\end{equation}
\noindent where $\text{diag}(\bm{\gamma}^2)$ is a diagonal matrix whose $i$-th diagonal element is $\gamma_i^2$, 
$\mathit{softmax}$ is the vector normalization: $\mathit{softmax}(\bm{v})_i = e^{v_i} / \sum_{j} e^{v_j}$, and is applied to matrices per row (that is, the softmax of a matrix is the matrix that has in its rows the softmax-normalized original rows), $\bm{I}_n$ is the order-n unit matrix, and $\circ$ is the elementwise matrix multiplication operator. 
\end{proposition}

\begin{proof}
\noindent Let $\bm{a}^T \in \mathbb{R}^n$, $\bm{\gamma} \in (\mathbb{R} \setminus \{0\})^n$, $\bm{Z} \in \mathbb{R}^{n^2}$. We verify that $\pi_{\alpha}(\bm{a}, \bm{\gamma}, \bm{Z}), \pi_{T}(\bm{a}, \bm{\gamma}, \bm{Z})$ is a valid PH in the Markovian representation. First, by the properties of the $\mathit{softmax}$ vector normalization function, $\pi_{\alpha}(\bm{a}, \bm{\gamma}, \bm{Z})$ contains only non-negative elements and sums to $1$ as required. Next, $\pi_{T}(\bm{a}, \bm{\gamma}, \bm{Z})$ has $-\bm{\gamma}_i^2$ in the $i$th diagonal element, and the non-diagonal $i,j$-th element is $\bm{\gamma}_i^2 \,\mathit{softmax}(\bm{Z})_{i,j}$, thus guaranteeing that the sum of rows of 
$\pi_{T}(\bm{a}, \bm{\gamma}, \bm{Z})$ are non-positive, since: $-\bm{\gamma}_i^2 + \bm{\gamma}_i^2 \cdot \sum_{j \neq i}^{n} \mathit{softmax}(\bm{Z})_{i,j} \leq 0$.

Conversely, let $\bm{\alpha}, \bm{T}$ be a PH in the Markovian representation with no zero elements. We use the property that elementwise-$\log$ is a right-inverse of $\mathit{softmax}$ over the domain of probability vectors with no zero elements.
Set $\bm{a} = \log(\bm{\alpha})$, $\bm{\gamma}_i = (-T_{i,i})^{\frac{1}{2}}$, 
$\bm{D}=diag(1/\bm{\gamma}^2) \bm{T}$,
$d_i=-\sum_{j,j\neq i} D_{ij}$,
$\bm{E}=\bm{D} +\bm{I} + diag(\bm{d})$,
and finally, set $\bm{Z}=\log(\bm{E})$.
Now by construction, $\pi_{\alpha}(\bm{a}, \bm{\gamma}, \bm{Z}) = \bm{\alpha}$, and $\pi_{T}(\bm{a}, \bm{\gamma}, \bm{Z}) = \bm{T}$. 
\end{proof}

The restriction of the re-parametrization to map onto the interior, the set Markovian PH representations with no zero elements, is a technicality induced by the positivity of the softmax function, but has no implications for numerical computability. The proposed method is in its essence an approximation, and while these boundary points of the PH distribution are not reachable, points arbitrarily near them are. For example, consider an $\bm{\alpha}$ vector with the elements $[1, 0, 0]$. Setting $\bm{a} = [100, 1, 1]$ yields $\mathit{softmax}(\bm{a}) \approx [1 - 2\cdot10^{-43}, 10^{-43}, 10^{-43}]$, which for all \textit{numerical} intents and purposes is precisely the target of $[1, 0, 0]$.

\begin{corollary}
Let $f(\bm{\alpha}, \bm{T})$ be a differentiable scalar function of the Markovian PH parameters, then:
$$f_{\pi}(\bm{a}, \bm{\gamma}, \bm{Z}) = f(\pi_{\alpha}(\bm{a}, \bm{\gamma}, \bm{Z}), \pi_{T}(\bm{a}, \bm{\gamma}, \bm{Z}))$$ 
is a differentiable function of the parameters $\bm{a}, \bm{\gamma}, \bm{Z}$. Consequently, local extremum points of any such function $f$ can be found using GD in the space of the new parametrization $(\bm{a}, \bm{\gamma}, \bm{Z})$. Once a point is found in the new parametrization, it can be converted back to the Markovian parameter space using the mapping $\pi$. 
Every Extremum point of $f$ clearly corresponds to (possibly multiple) extremum points of $f_{\pi}$ (via the pre-image of $\pi$), but the opposite direction is not guaranteed. This implies that it may be necessary to sample many starting points for the GD process. In practice, our experiments show that useful points in the Markovian parametrization are relatively easy to obtain in this method.   


\end{corollary}

\begin{corollary}
The moment fitting weighted regression function $\sum_{i=1}^{l} w_i \left( i! (-1)^i\bm{\alpha} \bm{T}^{-i} \mathds{1}_n - m_i \right)^2$ is a differential function of $\bm{\alpha}, \bm{T}$, and hence can be optimized in the space of the new parametrization using GD. The resulting point in the new parameterization is then guaranteed to be mapped back to a valid PH distribution in the Markovian parametrization. 
\end{corollary}

Using this parametrization, the moment matching optimization problem (Eq. \eqref{eq:opt-naive}) becomes: 

\begin{equation}
    \min_{\bm{a}, \bm{\gamma}, \bm{Z}}  \quad \sum_{i=1}^{l} w_i \left( i! (-1)^i \;\; \pi_{\alpha}(\bm{a}, \bm{\gamma}, \bm{Z}) \;\; \pi_{T}(\bm{a}, \bm{\gamma}, \bm{Z})^{-i} \;\; \mathds{1}_n \;\;- \;\; m_i \right)^2,
\label{eq:opt-unconstrained-genral}
\end{equation}
\noindent
with the only constraints being that all $\bm{\gamma}$ elements must be non-zero. This technical requirement of pointwise ineligibility has no practical implication in the GD optimization procedure so henceforth we consider this to be an unconstrained optimization problem. In the following we describe two additional re-parametrizations which limit the search to the Coxian and Hyper-Erlang subtypes of the PH distribution. 

\subsection{The Coxian re-parametrization}\label{sec:cox_representation}

The optimization problem defined in the previous section involves $\mathcal{O}(n^2)$ parameters for a general PH($n$) distribution. There are subclasses of PH($n$) distributions which can be described by much less, for example ($\mathcal{O}(n)$) parameters. These include for example PH distributions without a loop in the transition structure, referred to as acyclic PH distributions. These are known to have a Markovian representation with $2n-1$ parameters \cite{CUMANI}, as it is depicted in Figure \ref{fig:coxian}. This representation is characterized by 
\begin{equation}
\bm{\lambda} = 
\begin{bmatrix} \lambda_1 \\ \lambda_2 \\ \vdots \\ \lambda_n \\ 
\end{bmatrix}, \quad
\bm{p} = 
\begin{bmatrix} p_1 \\ p_2 \\ \vdots \\ p_{n-1} \\ 
\end{bmatrix}, 
\label{eq:parametrization-Coxian}
\end{equation}

where $\lambda_j>0$, $0\leq p_j<1$ for $\forall j\leq n$. Like for the general PH, the $i^{th}$moment of the Coxian distribution can be computed as: 
\begin{equation}
\label{eq:momcoxian}
   \mu^C_i(\bm{\lambda},\bm{p}) =
    i! (-1)^i \bm{\alpha} \bm{T}^{-i} \mathds{1}_n 
\end{equation}
assuming:
\begin{equation}
\begin{aligned} 
\bm{\alpha}^T &=     \begin{bmatrix} ~1~ \\  0 \\ 0 \\ \vdots \\ 0 \end{bmatrix}, \quad
\bm{T} = 
\begin{bmatrix} 
-\lambda_1 & p_1 \lambda_1  & 0 & \cdots & 0   \\ 
0 & -\lambda_2 & p_2\lambda_2  & \ddots & \vdots    \\ 
\vdots & \ddots & \ddots & \ddots & 0  \\
\vdots & 0 & \ddots & -\lambda_{n-1} & p_{n-1}\lambda_{n-1}   \\ 
0 & \cdots & \cdots & 0 & -\lambda_n     
\end{bmatrix}.
\end{aligned}
\label{eq:CoxianPH}
\end{equation}

The potential advantage of using special subclasses of PH distributions is the more efficient optimization with less parameters to optimize over, and its potential disadvantage is the reduced class of distributions in the search space. It is not known in advance which one of these two effects is dominant, and we investigate this question numerically in Section \ref{sec:experiment}.

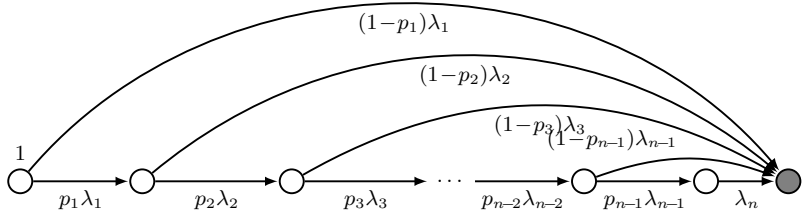
\begin{figure}
\centering
\resizebox{0.7\textwidth}{!}{
\begin{tikzpicture}[node distance=1.2cm]
\node[state,label=above:{$1$}] (S1) {};
\node[state,label=above:{$ $}] (S2) [right =14mm of  S1] {};
\node[state,label=above:{$ $}] (S3) [right =18mm of S2] {};
\node (S4) [right  =18mm of S3] {$\dots$};
\node[state,label=above:{$ $}] (S5) [right =14mm of S4] {};
\node[state,label=above:{$ $}] (S6) [right =14mm of S5] {};
\node[abs] (S0) [right of=S6] {};

\draw[tr] (S1) edge  node[below] {$p_1\lambda_1$} (S2);
\draw[tr] (S2) edge  node[below] {$p_2\lambda_2$} (S3);
\draw[tr] (S3) edge  node[below] {$p_3\lambda_3$} (S4);
\draw[tr] (S4) edge  node[below] {$p_{n\!-\!2}\lambda_{n\!-\!2}$} (S5);
\draw[tr] (S5) edge  node[below] {$p_{n\!-\!1}\lambda_{n\!-\!1}$} (S6);
\draw[tr] (S6) edge  node[below] {$\lambda_{n}$} (S0);

\draw[tr] (S1) edge  [bend left=50]    node[below] {$(1\!-\!p_1)\lambda_1$} (S0);
\draw[tr] (S2) edge  [bend left=40]    node[below] {$(1\!-\!p_2)\lambda_2$} (S0);
\draw[tr] (S3) edge  [bend left=30]    node[below] {$(1\!-\!p_3)\lambda_3$} (S0);
\draw[tr] (S5) edge  [bend left=20]    node[above left] {$(1\!-\!p_{n\!-\!1})\lambda_{n\!-\!1}$} (S0);
\end{tikzpicture}
}	
\caption{Coxian distribution, it is referred to as CF3 structure in \cite{CUMANI}}
\label{fig:coxian}
\end{figure}

For unconstrained optimization, we re-parametrize the Coxian distribution of size $n$, still using $2n-1$ parameters, as follows:
\begin{equation}
 \bm{\gamma} = 
\begin{bmatrix} \gamma_1 \\ \gamma_2 \\ \vdots \\ \gamma_n \\ 
\end{bmatrix}, \quad
\bm{u} = 
\begin{bmatrix} u_1 \\ u_2 \\ \vdots \\ u_{n-1} \\ 
\end{bmatrix},
\label{eq:re-parametrization-Coxian}
\end{equation}

\noindent with $\bm{\gamma} \in (\mathbb{R} \setminus \{0\})^n, \bm{u} \in \mathbb{R}^{n-1}$.

This is mapped onto the set of degree $n$ Coxian via the following transformation $\phi$:
\begin{equation}
    \bm{\lambda} = \phi_{\lambda}(\bm{\gamma}, \bm{u}) = \bm{\gamma^2}, \quad \bm{p} = \phi_{p}(\bm{\gamma}, \bm{u}) =    \sigma(\bm{u}),
\label{eq:parametrization-map-Coxian}
\end{equation}
\noindent where $\bm{\gamma^2}$ is computed elementwise, 
$\sigma$ is the sigmoid function, $\sigma(z) = 1 / (1 + e^{-z})$, and is used to map elements of $\bm{u}$ to the interval $(0, 1)$. As in the case of the softmax function, the point zero is not achievable with sigmoid, but points arbitrarily near it are, which is sufficient for numerical purposes.

\begin{proposition}
The image of $(\mathbb{R} \setminus \{0\})^n \times \mathbb{R}^{n-1}$ under $\phi$ is the set of valid Coxian distributions in the Markovian representation. That is, $\phi$ maps the parametrization (\ref{eq:re-parametrization-Coxian}) onto the set of Coxian distributions in the Markovian representation.
\end{proposition}

\begin{proof}
\noindent Let  $\bm{\gamma} \in (\mathbb{R} \setminus \{0\})^n, ~\bm{u} \in \mathbb{R}^{n-1}$. First, it is readable that $\bm{\gamma^2}$ is a valid intensity vector and $\sigma(\bm{u})$ is a vector of probabilities. 
Conversely, let $\bm{\lambda}, \bm{p}$ be a Coxian distribution in the Markovian representation. Then the $\phi$ mapping of  $\bm{\gamma} = \sqrt{\bm{\lambda}}$ and 
$\bm{u} = \sigma^{-1}(\bm{p})$ gives $\bm{\lambda}$ and $\bm{p}$.
\end{proof}

Using this parametrization, the moment-fitting weighted regression optimization problem (Eq. \eqref{eq:opt-naive}), restricted to the set of Coxian distributions with size $n$ becomes: 
\begin{equation}
    \min_{\bm{\gamma}, \bm{u}}  \quad \sum_{i=1}^{l} w_i \left( 
     \mu^C_i(\phi_{\lambda}(\bm{\gamma}, \bm{u}),\phi_{p}(\bm{\gamma}, \bm{u})) \;\;- \;\; m_i \right)^2.
\label{eq:opt-unconstrained-Coxian}
\end{equation}

\subsection{The Hyper-Erlang re-parametrization}\label{sec:mix_erlang_representation}

Hyper-Erlang distributions of order $n$ are a subset of acyclic PH distributions of order $n$ with $\mathcal{O}(k)$ parameters where $k<n$. That is Hyper-Erlang distributions further reduces the number of model parameters and further restricts the set of considered distributions compared to Coxian distributions of order $n$. A Hyper-Erlang distribution of order $n$ is composed by 
the mixture of $k$ Erlang distributions with sizes $d_1, ..., d_k$ where $n = \sum_i d_i$, as it is depicted in Figure \ref{fig:hypererlang}.

\begin{figure}
\centering
\begin{tikzpicture}
\node[state,label=left:{$\omega_1$}] (S10) {};
\foreach \from/\to in {S10/S11,S11/S12,S12/S13}
\node[state] (\to) [right =4.1mm of \from] {};

\node[state,label=left:{$\omega_2$}] (S20) [below =4mm of S10] {};
\node (S30) [below =4mm of  S20] {$\dots$};
\node[state,label=left:{$\omega_k$}] (S40) [below =4mm of S30] {};
\foreach \from/\to in {S40/S41,S41/S42}
\node[state] (\to) [right =4.1mm of \from] {};

\node[abs] (S0) [right =32.1mm of S30] {};

\foreach \from/\to in {S10/S11,S11/S12,S12/S13}
\draw[tr] (\from) edge  node[above =2mm] {$\lambda_1$} (\to);
\draw[tr] (S13) edge  node[above right =0mm] {$\lambda_1$} (S0);

\draw[tr] (S20) edge  node[above = 0mm] {$\lambda_2$} (S0);

\foreach \from/\to in {S40/S41,S41/S42}
\draw[tr] (\from) edge  node[below] {$\lambda_k$} (\to);
\draw[tr] (S42) edge  node[below =1mm] {$\lambda_k$} (S0);
\end{tikzpicture}
\caption{Hyper-Erlang distribution with $d_1=4$, $d_2=1$, \ldots, $d_k=3$.}
\label{fig:hypererlang}
\end{figure}
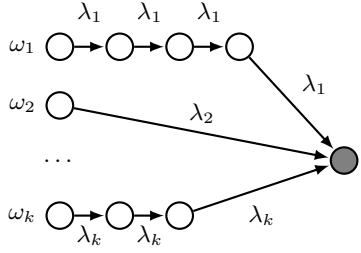

This representation is characterized by 
\begin{equation}
\bm{\omega} = 
\begin{bmatrix} \omega_1 \\ \omega_2 \\ \vdots \\ \omega_{k}  
\end{bmatrix}, \quad
\bm{\lambda} = 
\begin{bmatrix} \lambda_1 \\ \lambda_2 \\ \vdots \\ \lambda_k  
\end{bmatrix}, 
\label{eq:parametrization-Erlang}
\end{equation}
where $\lambda_j>0$, $0\leq \omega_j<1$ for $\forall j\leq k$, and
$\sum_{j=1}^k \omega_j =1$. The $i$-th moment
of the Hyper-Erlang distribution with sizes $d_1, ..., d_k$ and representation  $\bm{\omega},\bm{\lambda}$
can be computed as 
\begin{equation}
\label{eq:momerlang}
\mu^H_i(\bm{\omega},\bm{\lambda}) =
i! (-1)^i \bm{\alpha} \bm{T}^{-i} \mathds{1}_n 
\end{equation}
assuming
\begin{equation}
\begin{aligned} 
\bm{\alpha} &=   
\big[
\omega_1 \;,\;
 \underbrace{0, \cdots, 0}_{d_1-1} 
\;,\; \omega_2 \;,\;
\underbrace{0, \cdots, 0}_{d_2-1},
\cdots, \; \omega_k \;,\; 
\underbrace{0, \cdots, 0}_{d_k-1}
\big], 
\\
\bm{T} &= ~
\begin{array}{|ccccccc|c}
-\bm{\lambda}^2_1 & \bm{\lambda}^2_1 & 0 & \cdots & \cdots & \cdots & 0 & 1  \\ 
0 & \ddots & \ddots & \ddots &  &  & \vdots  &  \vdots\\ 
\vdots & \ddots & -\bm{\lambda}^2_1 & \bm{\lambda}^2_1 & \ddots &  & \vdots  & d_1 \\
\vdots &  & \ddots & \ddots & \ddots & \ddots & \vdots &  \vdots \\ 
\vdots &  &  & \ddots & -\bm{\lambda}^2_k & \bm{\lambda}^2_k & 0 & n\!-\!d_k\!+\!1\\ 
\vdots &  &  &  & \ddots & \ddots & \ddots &  \vdots \\ 
0 & \cdots & \cdots & \cdots & \cdots & 0  & -\bm{\lambda}^2_k & n \\ 
\end{array}
\end{aligned}
\label{eq:parametrization-matrix-Erlang}
\end{equation}
where the last column indicate the row indexes. 

For the purpose of unconstrained optimization, the Hyper-Erlang distribution with sizes $d_1, ..., d_k$ (where $n = \sum_i d_i$) is re-parametrized as follows:
\begin{equation}
\bm{\beta} = 
\begin{bmatrix} \beta_1 \\ \beta_2 \\\vdots \\ \beta_k \end{bmatrix},
\quad 
\bm{\delta} = 
\begin{bmatrix} 
\delta_1 \\ \delta_2 \\ \vdots \\ \delta_k  
\end{bmatrix}, 
\label{eq:re-parametrization-Erlang}
\end{equation}

\noindent where $\bm{\beta} \in \mathbb{R}^k$ relates to the probability of choosing each Erlang distribution, and $\bm{\delta} \in (\mathbb{R} \setminus \{0\})^k$ is related to the rate in each of the Erlang distributions. This is mapped to a 
Hyper-Erlang distribution with sizes $d_1, ..., d_k$
via transformation $\psi$ as:
\begin{equation}
\begin{aligned} 
    \bm{\omega} = \psi_{\omega}(\bm{\beta}, \bm{\delta}) =\mathit{softmax}(\bm{\beta}), \quad  
    \bm{\lambda} = \psi_{\lambda}(\bm{\beta}, \bm{\delta}) = \bm{\delta}^2. 
\end{aligned}
\label{eq:parametrization-map-Erlang}
\end{equation}

\begin{proposition}
The image of $\; \mathbb{R}^k \times (\mathbb{R} \setminus \{0\})^k$ under $\psi$ is the set of valid Hyper-Erlang distributions in the Markovian representation. That is, $\psi$ maps the parametrization (\ref{eq:parametrization-Erlang}) onto the set of Hyper-Erlang distributions in the Markovian representation. 
\end{proposition}

\begin{proof}
Let $\bm{\beta} \in \mathbb{R}^k, \bm{\delta} \in (\mathbb{R} \setminus \{0\})^k$. The vector 
$\bm{\omega} = \psi_{\omega}(\bm{\beta}, \bm{\delta}) =\mathit{softmax}(\bm{\beta})$
contains only non-negative elements and sums to $1$, 
and the vector
$\bm{\lambda} = \psi_{\lambda}(\bm{\beta}, \bm{\delta}) = \bm{\delta}^2$ contains only positive elements. 
Conversely, let $\bm{\omega}, \bm{\lambda}$ be a
Hyper-Erlang distributions in the Markovian representation. Then the $\psi$ mapping of  
$\bm{\beta} = \log(\bm{\omega})$ and $\bm{\delta} = \sqrt{\bm{\lambda}}$ gives $\bm{\omega}$ and $\bm{\lambda}$.
\end{proof}

Using this parametrization, the moment-matching weighted regression optimization problem (\ref{eq:opt-naive}), restricted to the set of Hyper-Erlang distributions with sizes $d_1, ..., d_k$ becomes: 
\begin{equation}
\min_{\bm{\beta}, \bm{\delta}}  \quad \sum_{i=1}^{l} w_i \left(\mu^H_i(\psi_{\lambda}(\bm{\beta}, \bm{\delta}),\psi_{p}(\bm{\beta}, \bm{\delta})) \;\;- \;\; m_i \right)^2.
\label{eq:opt-unconstrained-Erlang}
\end{equation}

\subsection{Beyond moment matching}\label{sec:beyond_moms}
The framework outlined above is useful for more general PH fitting problems as long as the fitting problem can be specified as differentiable functions over the parameters of the Markovian parametrization. Here we demonstrate the approach for general PH($n$) distributions, but it can be similarly applied for Coxian and Hyper-Erlang distributions as well following the same pattern of moment fitting in the previous sections. 
Let:
\begin{equation}
    f_i(\bm{\alpha}, \bm{T}) \quad \text{for} ~~i = 1, \ldots, l
\end{equation}
be differentiable functions describing properties of PH distribution with corresponding desired targets $g_i$. This leads to the following constrained weighted optimization problem:
\begin{align} 
\label{eq:opt-general-f-naive}
&\min_{ \bm{\alpha},\bm{T}}  \quad  \sum_{i=1}^{l} w_i \left( f_i(\bm{\alpha}, \bm{T}) - g_i \right)^2 \\ \nonumber
&\text{subject to:~} 
\forall i: \alpha_i \geq 0, \forall i,j \neq i: T_{ij} \geq 0, \sum_i \alpha_i = 1, \forall i: \sum_j T_{ij} \leq 0. 
\end{align}

\noindent This problem, instead, can be written in terms of the general reparameterization (Eq. \eqref{eq:general-re-parmetrization}) as: 

\begin{equation}
\label{eq:opt-general-f-params-pi}
    \min_{\bm{a}, \bm{\gamma}, \bm{Z}} \quad \sum_{i=1}^{l} w_i \left( 
    f_i\left(\pi_{\alpha}(\bm{a}, \bm{\gamma}, \bm{Z}), \; \pi_{T}(\bm{a}, \bm{\gamma}, \bm{Z})\right) - g_i 
    \right)^2,
\end{equation}
\noindent with the only constraints being that all $\bm{\gamma}$ elements must be non-zero. There are many examples of differentiable functions $f_i$ that describe useful properties that one might want to match when fitting a PH distribution. The moment matching objective (Eq. \eqref{eq:opt-unconstrained-genral}), for instance, is obtained by setting $f_i(\bm{\alpha}, \bm{T}) = i!(-1)^i\bm{\alpha} \bm{T}^{-i} \mathds{1}_n$, with the associated target value $g_i = m_i$ for each of the moments. Other immediately available properties include the PDF and CDF, as well as the hazard rate function at any given point or the Laplace transform function at any  real point. In some contexts it might be useful to mix these types of target properties; for example, one could use a set of functions and targets of mixed type to fit a PH distribution that follows a given shape over some interval, and also approximately conforms to a set of prescribed moments. 

\section{Fitting framework}\label{sec:framework}

So far, we have introduced the key aspect of our approach including the re-parametrizations for the general, Coxian, and Hyper-Erlang families of PH distributions, and the associated moment-fitting optimization problem for each one. In order to apply the method in practice, for a given set of moments, we need to first select one of the re-parametrizations, then select the size $n$ and hyper-parameters of the fitting process (such as step size and convergence criteria for the gradient descent procedure). 


In this paper, we do not attempt to optimize the hyper-parameter choice of which re-parametrization to use for a given case. Instead, we conduct experiments for all of them. In practice, given a sequence of moments, and without any prior information about the most appropriate structure to use, one could try all three and use the best-obtained result (or the one that gives a favorable trade-off between goodness of fit and number of parameters). 

For a given re-parametrization, we next need to decide the PH size $n$ in the case of the general and Coxian structures, and under the Hyper-Erlang, we require the block sizes $\{d_1, ...,d_k\}$. Finding the right $n$ or the right set $\{d_1, ...,d_k\}$, in order to fit a given sequence of moments, is not a trivial task. Clearly, it is always possible to use a larger than required size (there is an embedding of PH(n) in PH(n+1)), but opting for a large size is done at the cost of larger parameter spaces and computation time. 

In the experiments conducted here we attempted to fit each sequence of moments in several trials with different hyperparameter settings, and extract and report the best fit. This includes the value $n$, which represents the PH size. If we apply our approach to the case of Hyper-Erlang, the overall size is not sufficient, as it is required to specify the different sizes of all blocks, $\{d_1, ...,d_k\}$. 
Specifically, we considered $ n \in \{20,50,100\} $. The exact values of $\{d_1, ...,d_k\}$ chosen for each $n$ are given in Appendix~\ref{sec:appB}. We also experimented with selecting $n$ and $\{d_1, ...,d_k\}$ via Bayesian optimization~\cite{frazier2018tutorial}, but this approach did not yield improvements over the straightforward configurations presented here.

The performance of the proposed method is heavily influenced by the initial points selected for $(\bm{a}, \bm{T})$. To increase accuracy, decrease runtime, and reduce the dependence on the initial starting points, we apply the optimization in parallel with multiple starting points and stop when the best one converges or the maximal number of steps is reached. 

Suppose $s$ is the number of starting points that we initially use when executing the GD process. Increasing the value of $s$ stochastically increases the accuracy and the speed of convergence, since some initial points will have a better trajectory of convergence than others. However, this advantage comes at a computational cost that is linear in $s$. Via experimentation, we tuned $s$ and it was set to be 10000 in all the experiments presented in Section~\ref{sec:experiment}. We note that we set $s=10000$ only at the beginning of the GD procedure, and only the best processes continue until the end, with less successful trajectories terminated before the others converge. In Appendix~\ref{sec:appB}, we provide full details on the optimization scheme used, and specifically how many processes are kept alive as a function of the step number. 



A step (or epoch) is a single computation of the loss function  (c.f. \eqref{eq:opt-unconstrained-genral}) and the ensuing gradient descent update of the Markovian representation for each of the active copies of the optimization problem. We optimize over $125000$ epochs at most. In order to reduce computation time, we added additional stopping criteria. If the loss function is less than a pre-determined $\epsilon = 10^{-9}$ \footnote{Empirically, this has been shown to produce extremely low errors.}, we stop the process entirely and return the current best PH distribution. Such a value suggests that the optimization procedure solves with a negligible error.



\section{Experiments}
\label{sec:experiment}

The method presented in this paper relies on non-convex optimization, hence convergence to the global minimum is not guaranteed. In this section we evaluate the method empirically. In order to properly evaluate any moment fitting method, it is necessary to test it against a wide range of moment targets, reflecting different types of distributions. To this end, we propose a diverse data set to assess how closely the moments of the PHs obtained by the method match the target moments. In Section~\ref{sec:testset}, we provide full details of the test set used for this purpose, highlighting its diversity. Section~\ref{sec:peva} presents a performance evaluation aiming to measure how effectively the method fits the moments in the test set. 

\subsection{Test set}\label{sec:testset}

This section provides the details of the generation process for the sequences of moment with which we test our method. For this purpose, we sample PH distribution from three different PH structures, as follows: 


\subsubsection{Sampling procedure}\label{sec:sampling}

We consider three sampling structures for a given PH size $n$. We first sample the size of a PH uniformly from $[1, 200]$ and than sample the parameters of the distribution according to the following cases. 

\begin{itemize}
    \item \textbf{General}: We sample the $\bm{\alpha}$, $\bm{\gamma}$, $\bm{Z}$ parameters of the re-parametrization (Eq. \eqref{eq:general-re-parmetrization}) and map the result to the Markovian PH representation via the mapping $\pi$ (Eq. \eqref{eq:parametrization-map}). We generated $\bm{\alpha}$ to be sampled uniformly in the interior of the $n$-dimensional simplex. The values of $\bm{\gamma}$ are sampled uniformly from the interval $[0.1, 10]$. The values of $Z$ are sampled uniformly from the interval $[0, 1]$, and to broaden the range of moments obtained by this procedure, we then set some of the non-diagonal elements of $\bm{T}$ to zero. The number of elements set to zero is sampled uniformly between $0$ and $n^2 - n$  with their locations chosen uniformly. A Markovian PH with parameters $\bm{\alpha}, \bm{T}$ is then obtained via the mapping $\pi$. 
    \item \textbf{Coxian}: Here we sample a Coxian distribution complying with the structure described in Section~\ref{sec:cox_representation}, where $p_i$, for $i \leq n$  is sampled uniformly from 0 to 1, and $\lambda_i$ is sampled uniformly from the interval of 0.1 to 10, for every $i\leq n$. 
    \item \textbf{ Hyper-Erlang}: Here, we sample a Hyper-Erlang distribution complying with the structure described in Section~\ref{sec:mix_erlang_representation}. We sample $\omega_i$ and $\lambda_i$ uniformly from the intervals $(0,1)$ and $(0.1,10)$ respectively, for $i \leq k$, where $k$, the number of blocks, is sampled from the interval $[1, \frac{n}{2}]$, which allows room for multiple blocks while maintaining them with decent sizes. The values $d_1,\ldots,d_k$, are chosen arbitrary, such that $d_i \geq 1, \forall i \leq k $. This is equivalent to choosing $k-1$ among  $n - 1$ units. 
\end{itemize}

 We note that both Coxian and Hyper-Erlang are known to be a large subsets of PH family and hence are widely used for moment matching~\cite{10.1007/978-3-642-39408-9_17, OSOGAMI2006524, doi:10.1287/opre.30.1.125, Altiok01061985}.
We emphasize that in the test set generation, the Coxian and Hyper-Erlang structures are used not because they involve a small number of parameters, but because they encompass a broad subset of distributions. For completeness, we also include samples from the general structure, that covers, in principle, all the space of PH distributions. Without loss of generality, we scale each PH distribution so that its first moment is exactly 1 (this is just a selection of convenient units). From each PH structure, we generated 500 examples.

\subsubsection{Testset coverage}\label{subsec:testset}

Using the data set generated from the sampling framework described, we present the range of each order of the first 20 moments in our test set in Table~\ref{tab:mom_range} together with the $25^{th}$ and $75^{th}$ percentiles. In Figures~\ref{fig:scv_skewness} and~\ref{fig:scv_kurtosis}, we illustrate various combinations of Squared Coefficient of Variance (SCV) values alongside the corresponding skewness and kurtosis values, respectively. These figures highlight the joint diversity of the first four moments, given that the first moment is always set to 1. From Figures~\ref{fig:scv_skewness} and~\ref{fig:scv_kurtosis}, it is evident that each sampling structure contributes to the diversity of the test set's moment space.
From the subset relations 
\[ \text{Hyper-Erlang PH($n$)} \subset \text{Coxian PH($n$)} \subset \text{General ($n$)} \]
one would expect that the General PH($n$) samples cover the whole valid ranges in Figures~\ref{fig:scv_skewness} and~\ref{fig:scv_kurtosis}, but the way we sample the General PH($n$) distributions has an effect on the moments. As a result the samples from the three subsets (General, Coxian, Hyper-Erlang) seem to be dominant in different regions in Figure~\ref{fig:scv_kurtosis}.

\begin{table}[!htp]\centering
\caption{Statistic properties of test set moments}\label{tab:mom_range}
\scriptsize
\begin{tabular}{rrrrr||rrrrrr}\toprule
Moment &min &max &25\% &75\% &Moment &min &max &25\% &75\% \\
\hline
1 &1 &1 &1 &1 &11 &1.31037 &7.84e+48 &2.29e+8 &6.37e+15 \\
2 &1.005 &384.63 &1.84 &5.78 &12 &1.382441 &1.16e+56 &3.15e+9 &7.38e+17 \\
3 &1.015 &355172.7 &6 &72.37 &13 &1.465387 &1.85e+63 &4.83e+10 &1.08e+20 \\
4 &1.03 &4.81e+8 &26.59 &1570.974 &14 &1.560637 &3.20e+70 &8.82e+11 &1.84e+22 \\
5 &1.051 &1.81e+12 &149.76 &51060.71 &15 &1.669882 &5.91e+77 &1.55e+13 &3.85e+24 \\
6 &1.077 &1.02e+17 &1111.79 &2262321 &16 &1.795123 &1.16e+85 &3.13e+14 &7.50e+26 \\
7 &1.109 &6.21e+22 &9616.98 &1.22e+8 &17 &1.938733 &2.44e+92 &5.79e+15 &1.64e+29 \\
8 &1.148 &4.33e+28 &95385.08 &8.21e+9 &18 &2.103525 &5.41e+99 &1.23e+17 &3.83e+31 \\
9 &1.194 &4.70e+34 &1065126 &6.43e+11 &19 &2.292842 &1.27e+107 &2.75e+18 &1.03e+34 \\
10 &1.248 &5.78e+41 &15617661 &5.73e+13 &20 &2.510662 &3.12e+114 &5.12e+19 &2.77e+36 \\
\bottomrule
\end{tabular}
\end{table}



\begin{figure}[ht]
    \centering
    \begin{subfigure}[b]{0.48\textwidth}
        \centering
        \includegraphics[width=\textwidth]{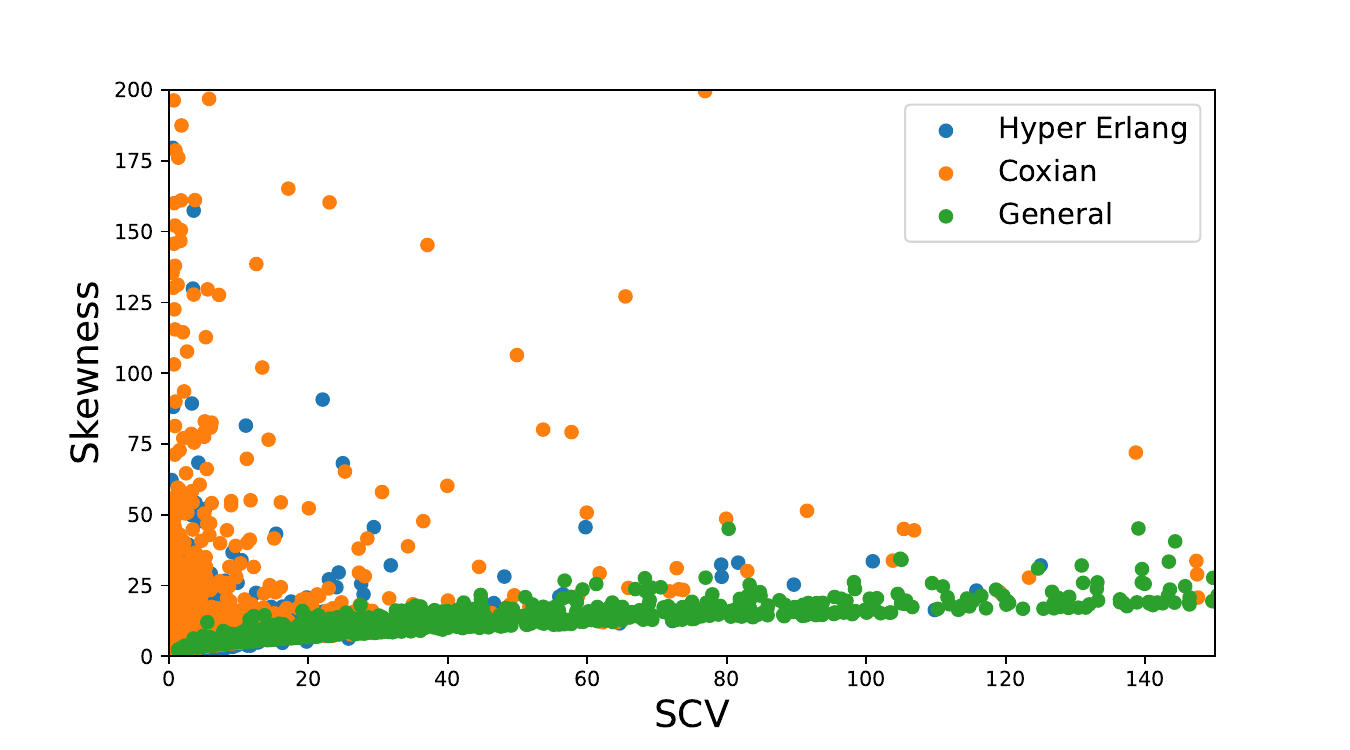}
        \caption{SCV-Skewness combination by different sampling techniques. }
        \label{fig:scv_skewness}
    \end{subfigure}
    \hfill
    \begin{subfigure}[b]{0.48\textwidth}
        \centering
        \includegraphics[width=\textwidth]{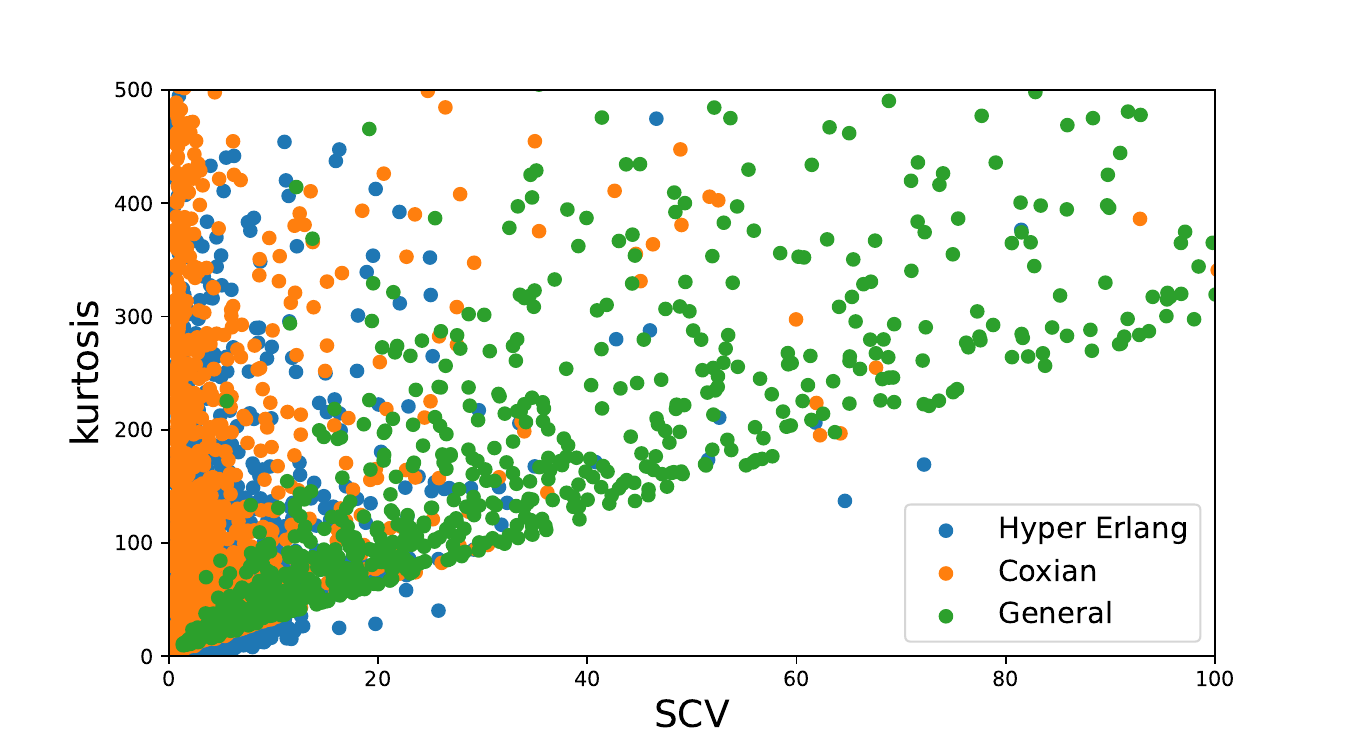}
        \caption{SCV-Kurtosios combination by different sampling techniques}
        \label{fig:scv_kurtosis}
    \end{subfigure}
    \caption{SCV, Skewness, and Kurtosis scatter.}
    \label{fig:scv_analysis}
\end{figure}

\subsection{Performance Evaluation}\label{sec:peva}

The main analysis we present below aims to determine how the accuracy and runtime of the proposed method are affected by two parameters. The first is the number of fitted moments. Typically, the more moments are fitted, the more challenging the task. The second parameter is the size $n$ of the PH used. Naturally, a larger $n$ will be more likely to match moments accurately. At least, for a given and achievable set of moments there is a minimal (yet unknown) size required. However, increasing the size will incur longer runtimes. In our experiments, we examine these questions by varying the following parameters: 

\begin{itemize}
\item $l$ - number of fitted moments. We examine accuracy and runtime for the first 5, 10, and 20 moments, i.e., $[1,\ldots,5], [1,\ldots,10], [1,\ldots,20]$ moments.
\item $n$ - the fitted PH size. We examine accuracy and runtime for sizes 20, 50, and 100. 
\end{itemize}

For a specific instance to be considered accurate, we suggest the following metric. We first compute the Mean Absolute Percentage Error (MAPE) for each individual moment. Formally,

\[
MAPE_i(j) =\left| \frac{m_i(j) - \hat{m}_i(j)}{m_i(j)} \right| \times 100,
\]

\noindent where $m_i(j)$ and $\hat{m}_i(j)$  are the target and obtained value of the $i^{th}$ moment for the $j^{th}$ instance. Instance $j$ is then considered accurate  if :

$$\max_{i\leq M} MAPE_i(j) \leq \eta,$$ 

\noindent for some pre-determined threshold $\eta$. Finally, the reported accuracy of the method in each setting is measured by the \textit{success rate}, which is defined as: 

\begin{align*}
    \text{success rate}[\%] = 100\times\frac{\text{Number of accurate instances}}{\text{Total number of instances}}.
\end{align*}

In the next section, we present runtime and the \textit{success rate} for each part of the test set described in Section~\ref{sec:sampling}, and for each considered combination of number of moments $M$ and the PH size $n$, and for the three re-parameterization (i.e., General, Coxian, and Hyper-Erlang). Success rate and runtime results are presented separately for each value of $\eta \in \{0.2\%,0.5\%,1\%\}$. In all moment fitting experiments we use for the weights $w_i = m_i^{-2}$ in order to make the optimization process scale agnostic and treat all moments equally. All experiments reported in this paper were preformed on an Intel Xeon Gold 5115 Tray Processor with 128 GB RAM.

\section{Results}\label{sec:results}


In this section, we present the results of the numerical experiments. The main accuracy results are presented in nine figures devoted to the combinations of test data set (the three sampling procedures discussed in Section \ref{sec:sampling}) and threshold values $\eta \in \{0.2\%, 0.5\%, 1\%\}$ defining success rates. Within each figure, we present the success rate of the moment fitting method for each combination of $l$ (i.e., the number of fitted moments), $n$, the PH size used for fitting, and the re-parametrization used. The values of $l$ and $n$ are marked in the x-axis ticks, and the re-parametrization is given by the color of the bar chart, as described in the legend. 

In Figures~\ref{fig:cox1},~\ref{fig:gen1} and~\ref{fig:hyp1}, we present the results for the threshold of $\eta = 1\%$, for the General, Coxian, and Hyper-Erlang data sets, respectively. In Figures~\ref{fig:cox05},~\ref{fig:gen05} and~\ref{fig:hyp05}, we present the results for $\eta =0.5\%$, for the General, Coxian, and Hyper-Erlang data sets, respectively. Finally, in Figures~\ref{fig:cox02},~\ref{fig:gen02} and~\ref{fig:hyp02}, we present the $\eta = 0.2\%$ results for the General, Coxian, and Hyper-Erlang data sets, respectively.

Generally speaking, these results indicate the model achieves an excellent success rate in fitting moments. It is near-perfect in many tested conditions, especially for the general and Coxian datasets and when fitting up to $10$ moments. The model performance is the worst when trying to fit 20 moments under the Hyper-Erlang data set, as is demonstrated in Figure~\ref{fig:hyp02}.

We observe two intuitive trends. First, the success rate decreases as the number of moments fitted increases. Second, allowing a larger PH size leads to higher accuracy and an improved success rate. We also observe that the model performs best under the general data set, for all values of $\eta$. We hypothesize that this dataset, unlike the other two, is more suitable for the method proposed in this paper because a greater variety of parameter combinations can represent these distributions. This is due to their origin from the general PH structure, which offers more degrees of freedom. Consequently, this may lead to a larger number of local minimum points in the parameter space being optimized, providing more convergence possibilities for the process. 

In most cases, the Hyper-Erlang re-parametrization performs better than the Coxian re-parametrization, which in turn performs better than the general one, in terms of accuracy. Surprisingly, even the instances that were originated from the Coxian dataset were fitted more accurately by the Hyper-Erlang re-parametrization. 

We suspect, that the larger number of parameters in the general case, results in more local minima in the goal functions, out of which many results in low accuracy. As a result, the GD optimization gets captured by an inaccurate local minimum with high probability, and, in spite of the large number of initial points, low accuracy is obtained. In contrast, the lower number of parameters of the Coxian and the Hyper-Erlang
cases result in goal functions with a lower number of local minima, and this way, the large number of initial points allows for exploring accurate minima with higher probability.

\begin{figure}[ht!]
    \centering

    \begin{subfigure}{0.3\textwidth}
        \centering
        \includegraphics[width=\textwidth]{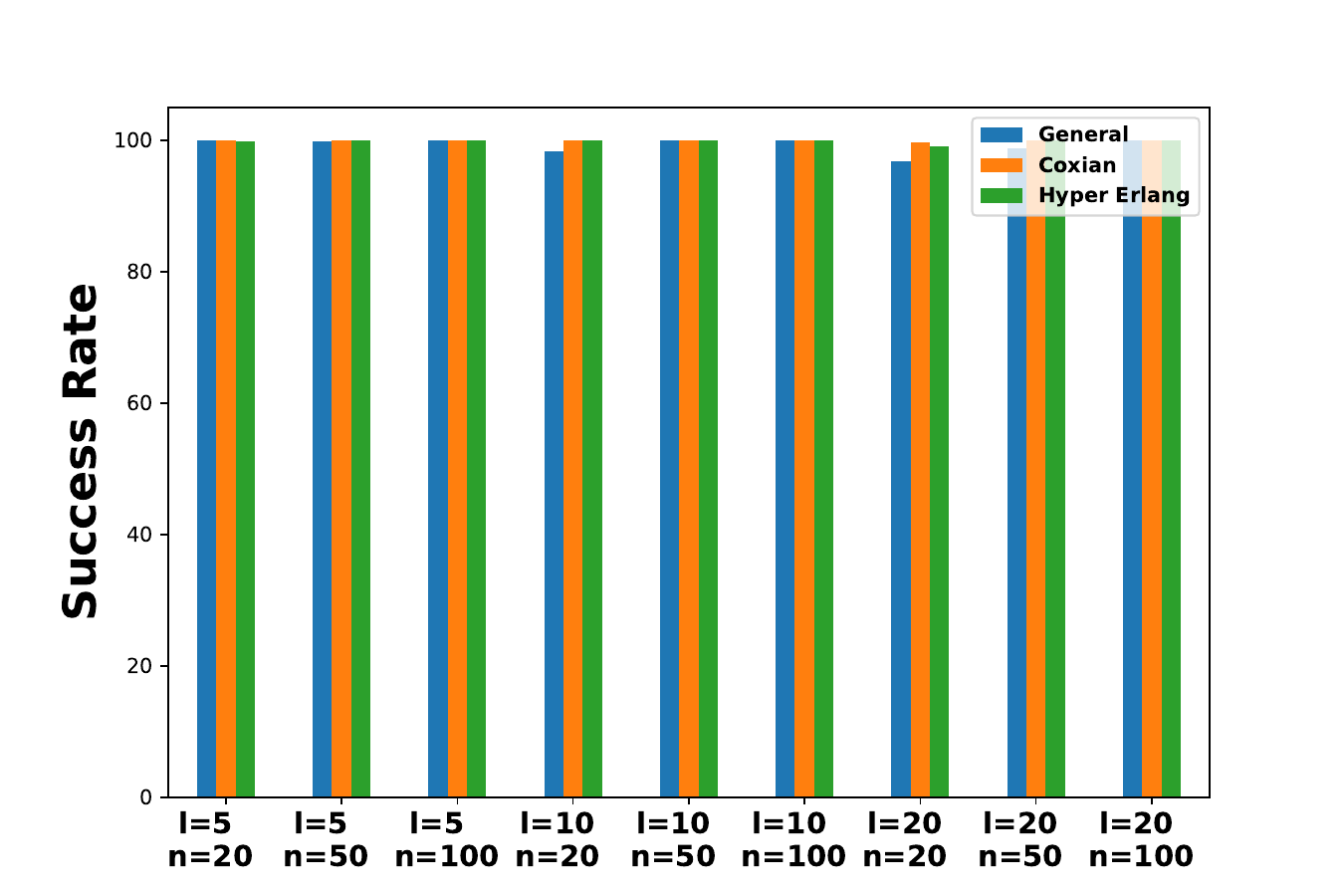}
        \caption{General data set}
        \label{fig:gen1}
    \end{subfigure}
    \hfill
    \begin{subfigure}{0.3\textwidth}
        \centering
        \includegraphics[width=\textwidth]{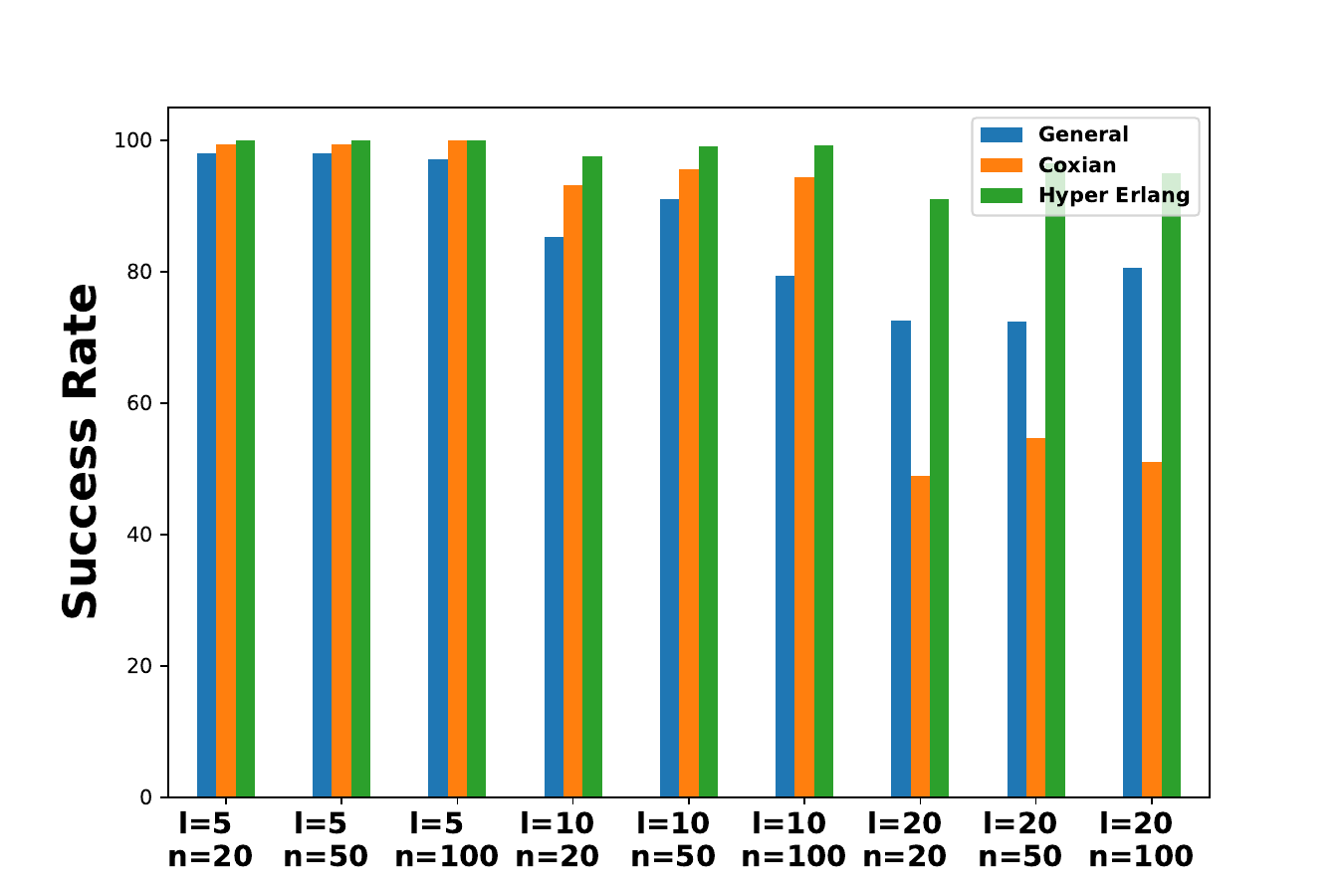}
        \caption{Coxian data set}
        \label{fig:cox1}
    \end{subfigure}
    \hfill
    \begin{subfigure}{0.3\textwidth}
        \centering
    \includegraphics[width=\textwidth]{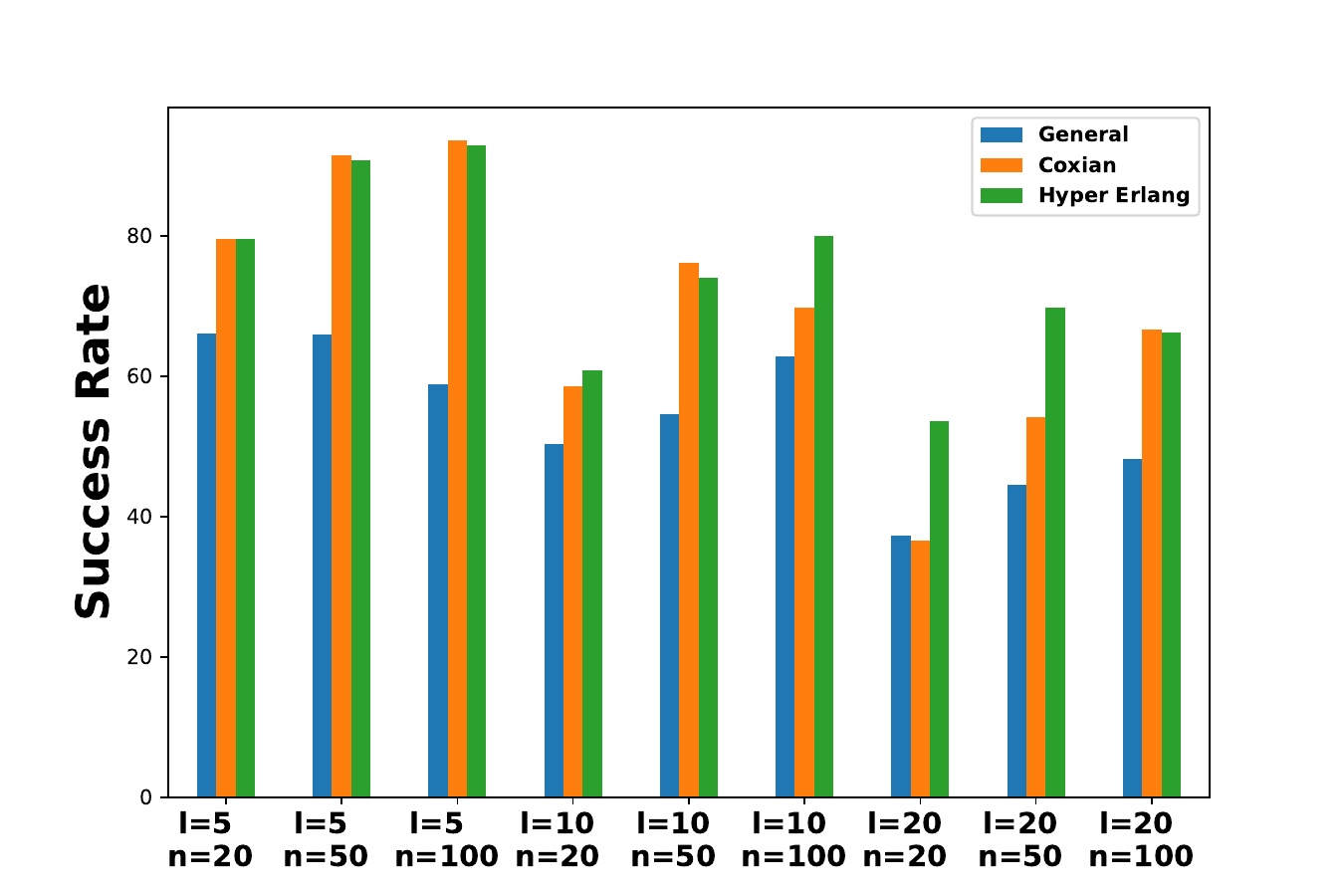}
        \caption{Hyper-Erlang data set}
        \label{fig:hyp1}
    \end{subfigure}

    \caption{Success rate $[\%]$ with fitting threshold $\eta=1\%$ (a fit attempt is deemed a success if all moments are within $1\%$ of their respective targets). X-axis labels describe the number of fitted moments $l$ for each bar and the maximum allowed PH size $n$ (short for $n$).}
    \label{fig:three_images_again}
\end{figure}

\begin{figure}[ht!]
    \centering

    \begin{subfigure}{0.33\textwidth}
        \centering
        \includegraphics[width=\textwidth]{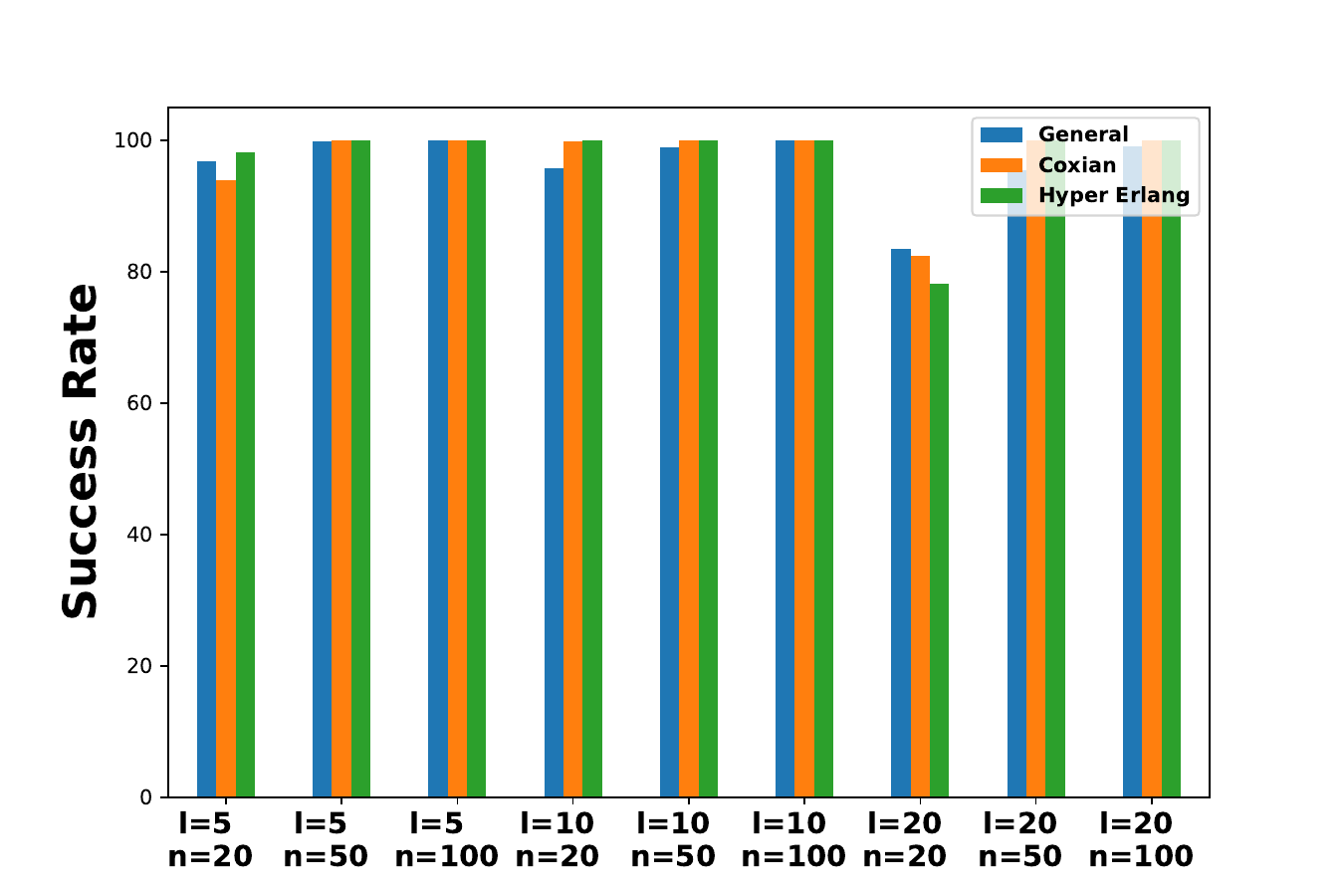}
        \caption{General Data set}
        \label{fig:gen05}
    \end{subfigure}
    \hfill
    \begin{subfigure}{0.3\textwidth}
        \centering
        \includegraphics[width=\textwidth]{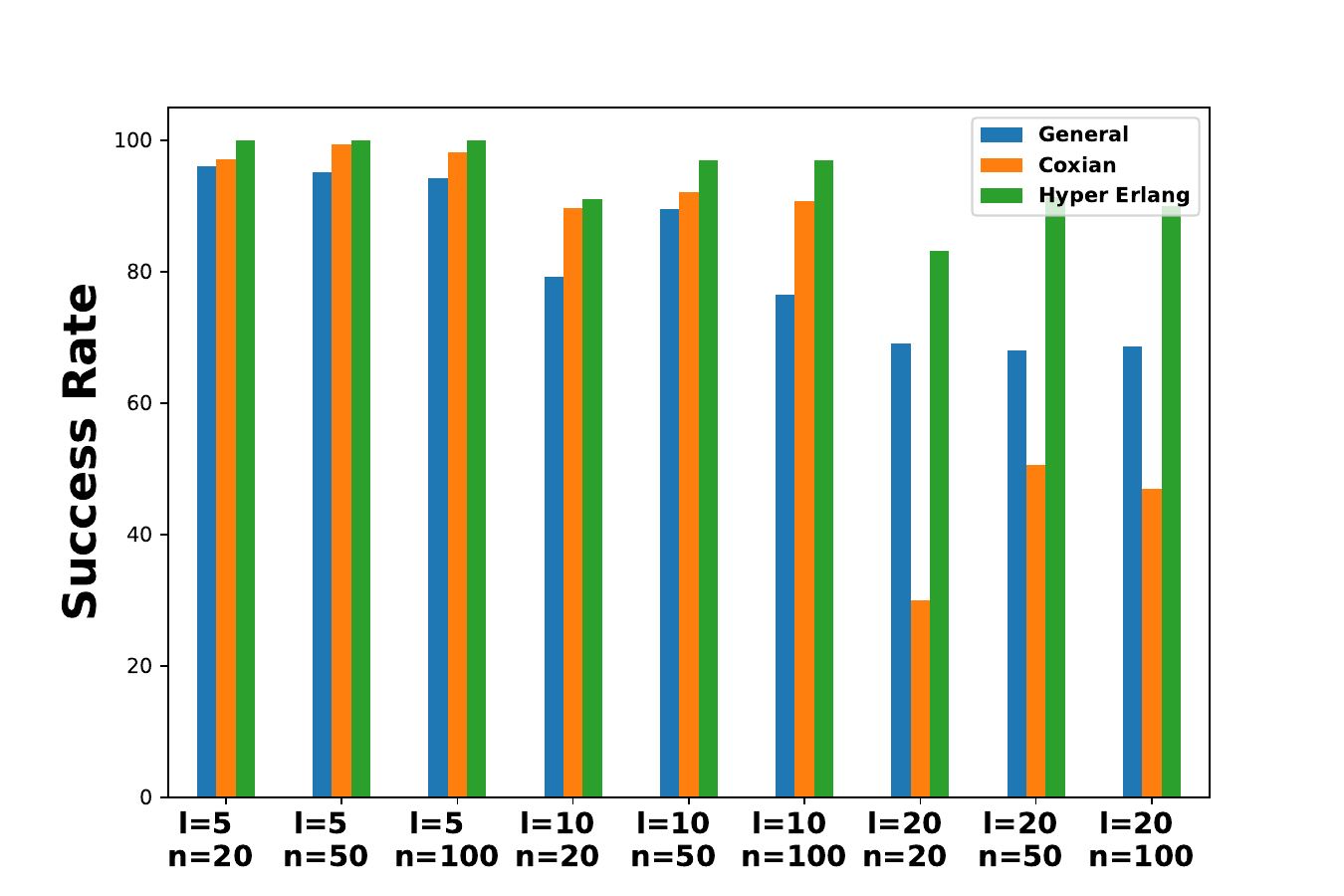}
        \caption{Coxian Data set}
        \label{fig:cox05}
    \end{subfigure}
    \hfill
    \begin{subfigure}{0.3\textwidth}
        \centering  \includegraphics[width=\textwidth]{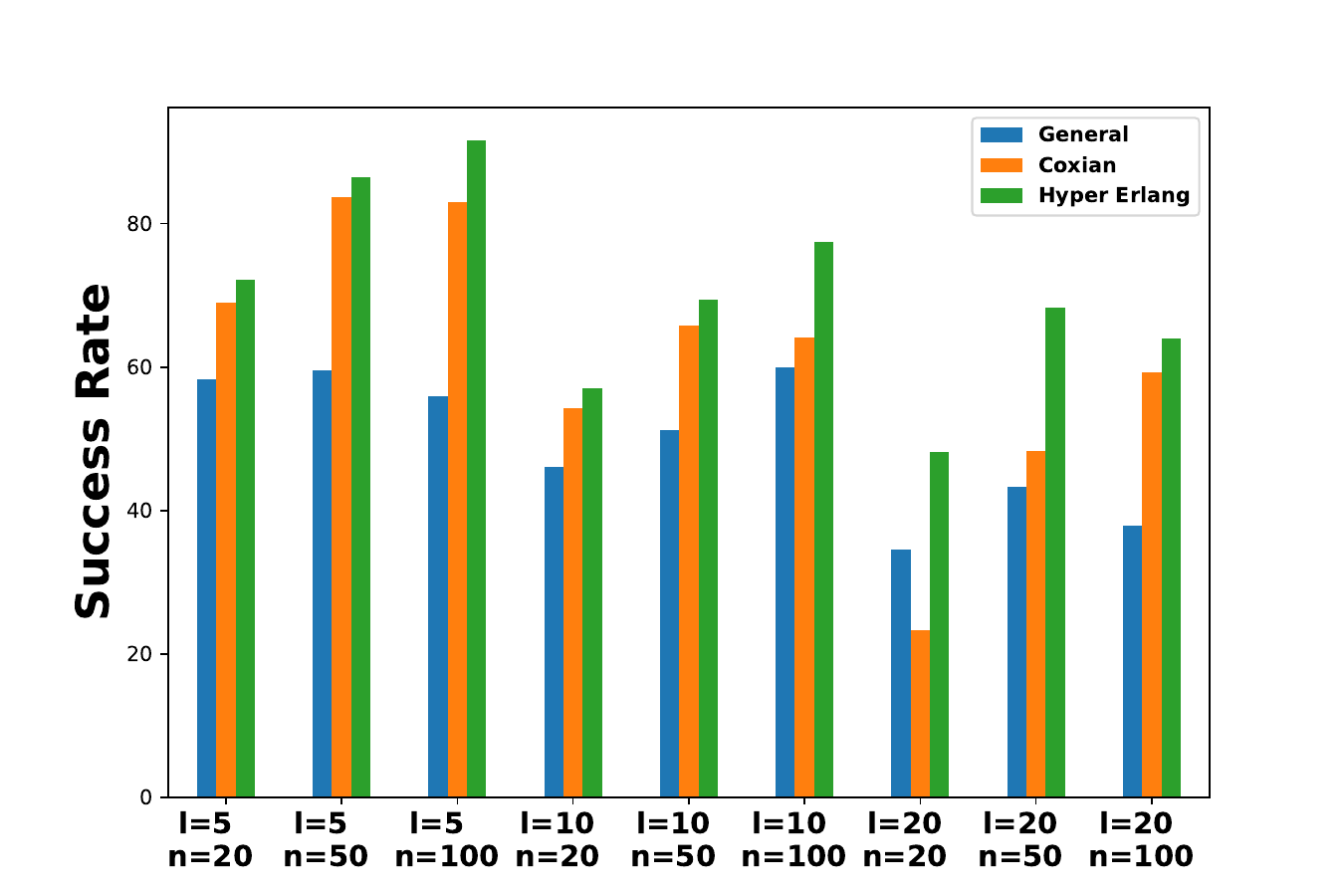}
        \caption{Hyper-Erlang data set}
        \label{fig:hyp05}
    \end{subfigure}

    \caption{Success rate $[\%]$ with fitting threshold $\eta=0.5\%$. See caption of Figure \ref{fig:three_images_again} for details.}
    \label{fig:three_images}
\end{figure}

\begin{figure}[ht!]
    \centering

    \begin{subfigure}{0.3\textwidth}
        \centering
        \includegraphics[width=\textwidth]{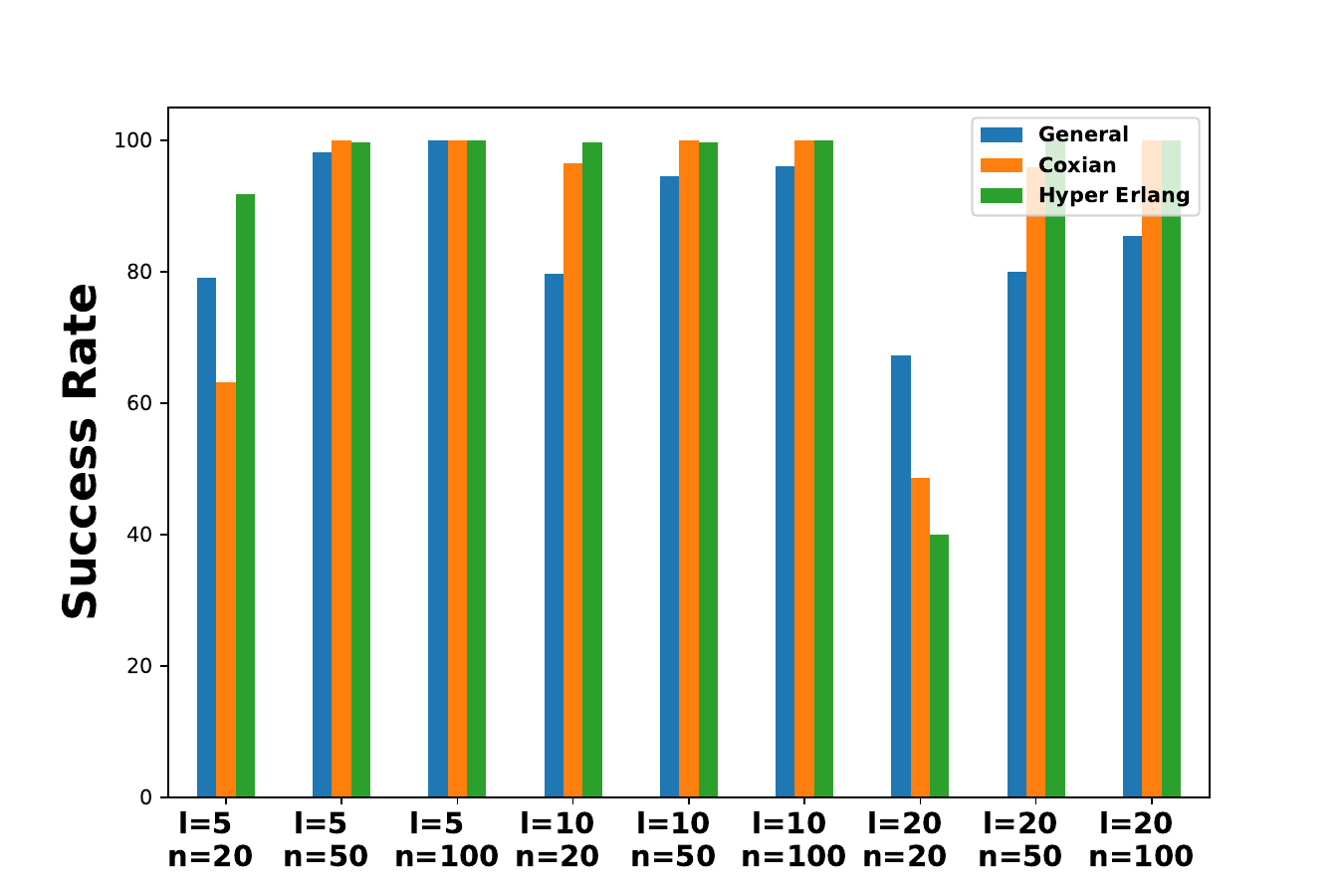}
        \caption{General data set}
        \label{fig:gen02}
    \end{subfigure}
    \hfill
    \begin{subfigure}{0.33\textwidth}
        \centering
        \includegraphics[width=\textwidth]{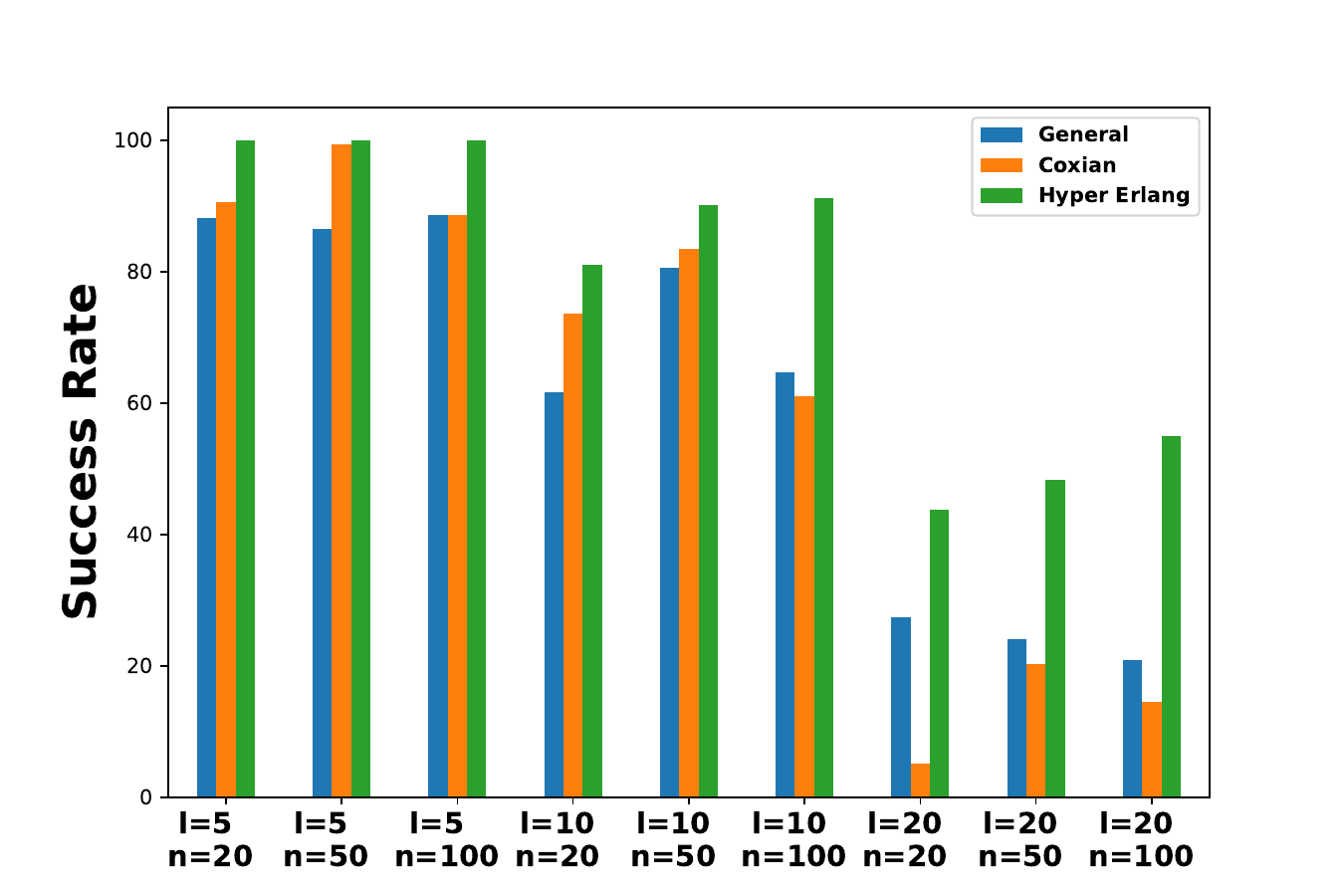}
        \caption{Coxian Data set}
        \label{fig:cox02}
    \end{subfigure}
    \hfill
    \begin{subfigure}{0.33\textwidth}
        \centering
        \includegraphics[width=\textwidth]{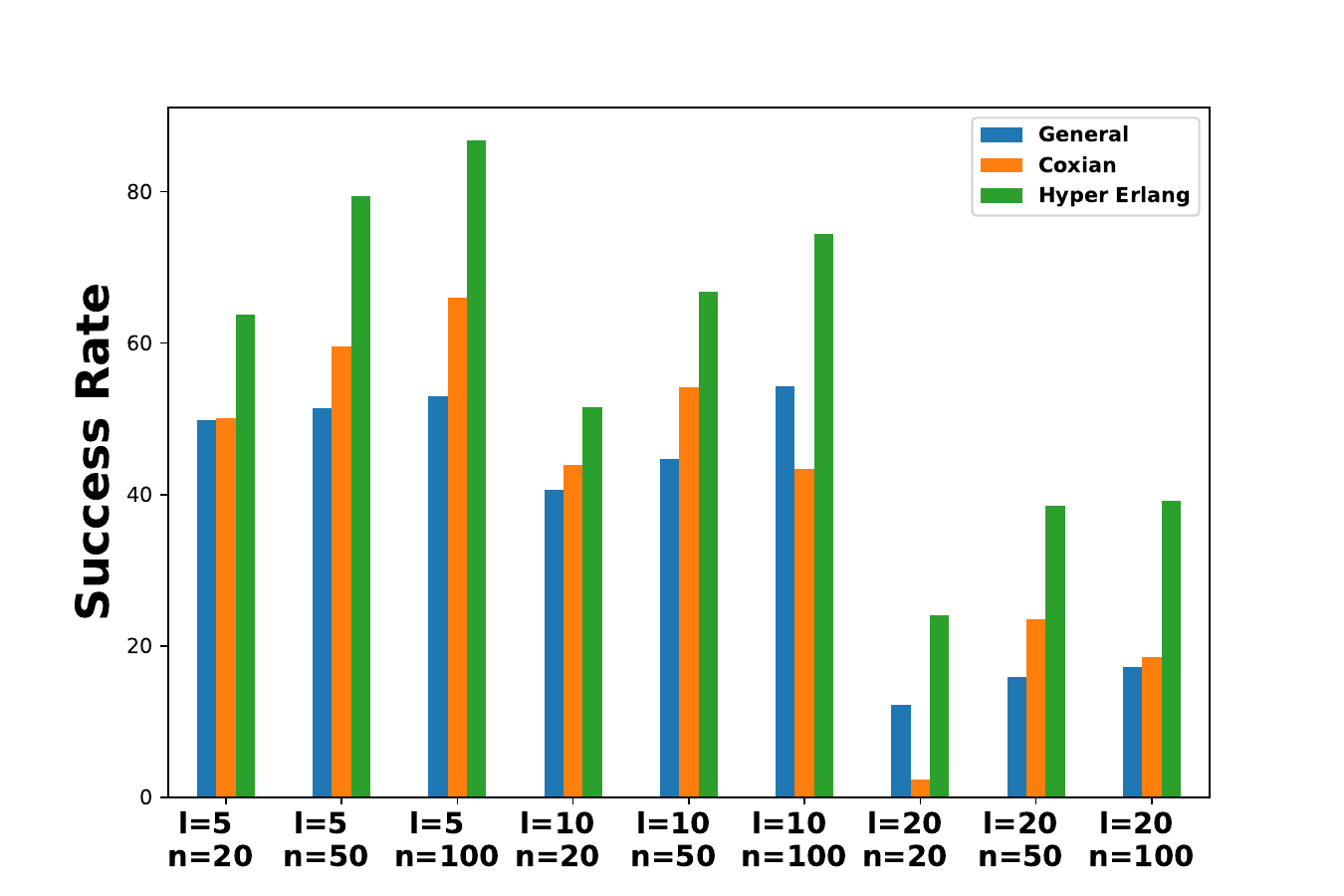}
        \caption{Hyper-Erlang data set}
        \label{fig:hyp02}
    \end{subfigure}

    \caption{Success rate $[\%]$ with fitting threshold $\eta=0.2\%$. See caption of Figure \ref{fig:three_images_again} for details.}
    \label{fig:02}
\end{figure}

In Figures~\ref{fig:runtimes_cox},~\ref{fig:runtimes_gen} and~\ref{fig:runtimes_hyp}, we present the runtimes in hours of the fitting methods on the General, Coxian,  and Hyper-Erlang data sets, respectively. The average runtimes range from minutes to three hours. The longer runtimes occur  under the Hyper-Erlang test data set.  

Typically, the more moments we fit, and the larger the PH we allow, the longer it takes. The results demonstrate a natural trade-off between runtimes and accuracy as we increase the PH size $n$. Some cases exhibit a counter-intuitive phenomenon where increasing $n$ comes with a runtime decrease, as shown in all figures. We hypothesize that this happens when we reach the stopping criterion (i.e,  the loss function is lower than $10^{-9}$ ). This might explain non-monotone behavior as well. Where fitting with $n=50$ takes longer than $n=20$ because the computations are heavier, but $n=100$ takes less time, since the PH is large enough such that it converges quickly to the solution and the procedure stops after fewer steps. This can be due to many reasons, such as a more frequent early stopping and better convergence for larger $n$. 

\begin{figure}[ht!]
    \centering
    \begin{subfigure}{0.3\textwidth}
        \centering
        \includegraphics[width=\textwidth]{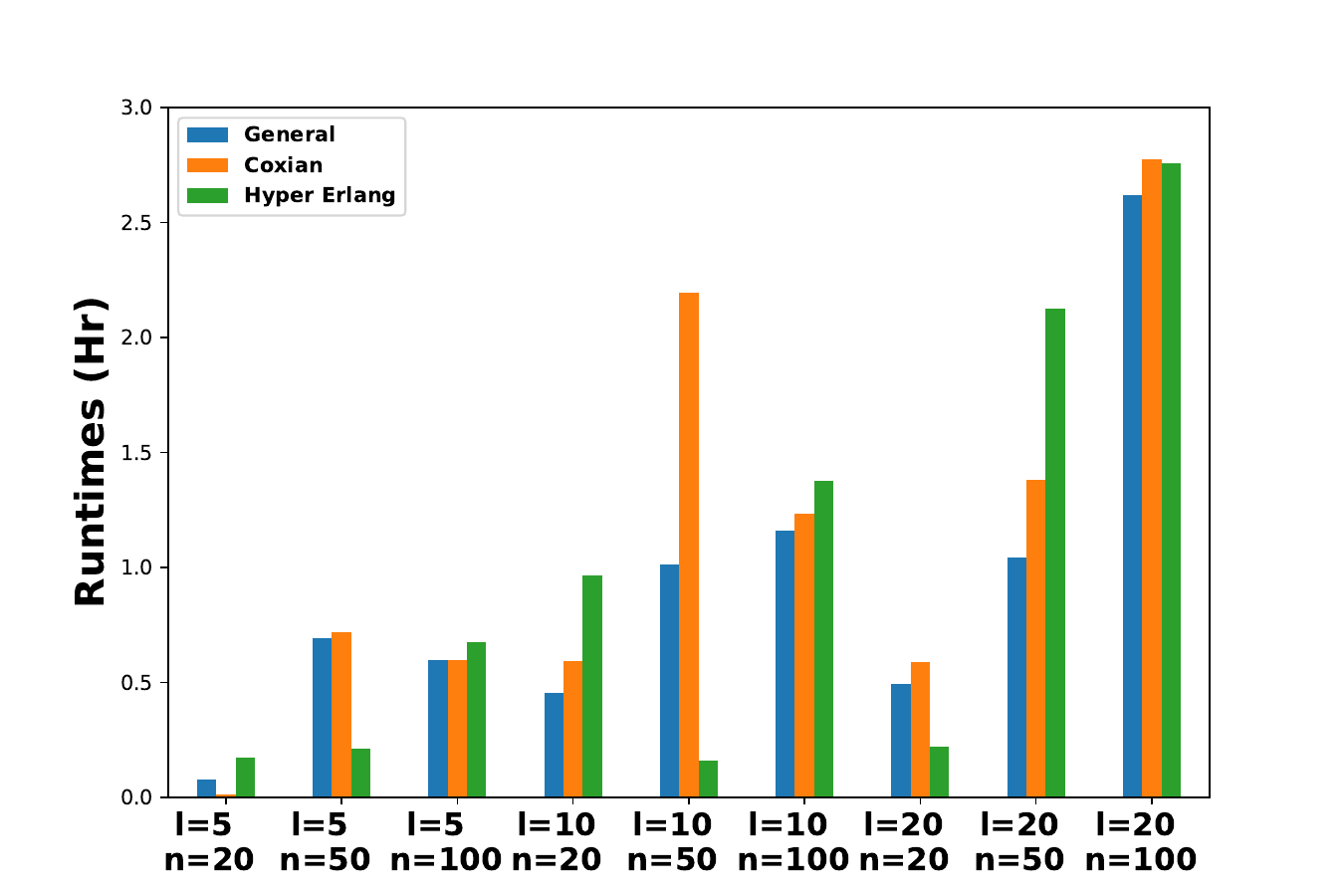}
        \caption{General data set}
        \label{fig:runtimes_gen}
    \end{subfigure}
    \hfill
    \begin{subfigure}{0.3\textwidth}
        \centering
\includegraphics[width=\textwidth]{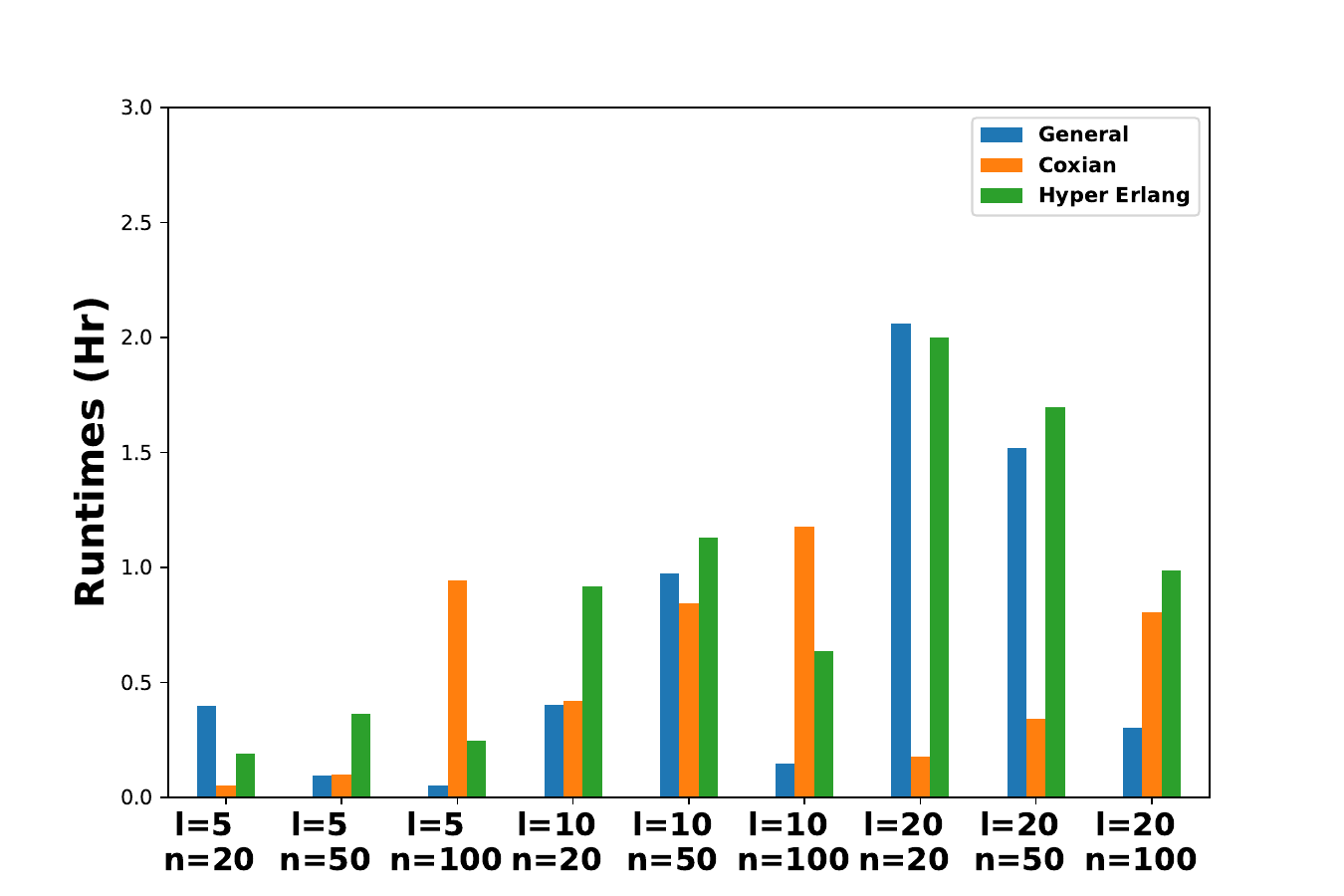}
        \caption{Coxian data set}
        \label{fig:runtimes_cox}
    \end{subfigure}
    \hfill
    \begin{subfigure}{0.3\textwidth}
        \centering
        \includegraphics[width=\textwidth]{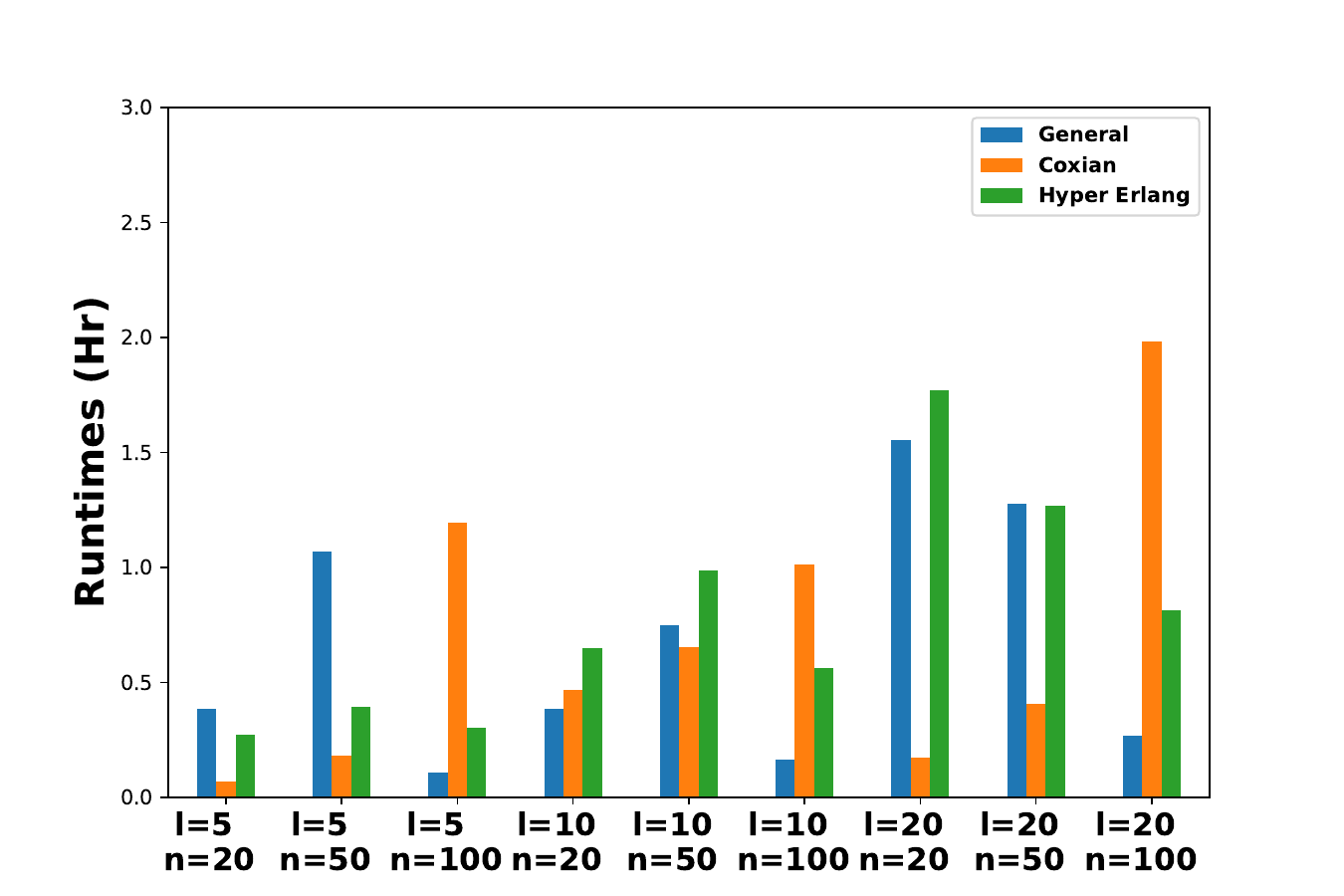}
        \caption{Hyper-Erlang data set}
        \label{fig:runtimes_hyp}
    \end{subfigure}

    \caption{Mean run time [Hr] for fitting of each test dataset and each experimental setting.}
    \label{fig:runtimes}
\end{figure}

\section{Joint fitting of moments and shape}
\label{sec:shape}

The method presented in this paper is not restricted to moment fitting (see also Section~\ref{sec:beyond_moms}). Any differentiable property of the Markovian PH representation can be added to the optimization procedure without changing anything essential in the method. Hence, also the shape of the distribution can be fitted simultaneously with the moments. For this purpose, it is required to adjust the objective function to include this component. The shape can be fitted via the PDF or CDF. 
To avoid particular dependency on the shape of the PDF and the considered points of fitting, we demonstrate the joint fitting ability of the proposed procedure using the points of the CDF. 

Suppose that as well as the list of moments, one is interested to obtain a PH distribution such that a set of points $\{(x_j, y_j)\} \in CDF$ ($j=1,\ldots,J$), meaning 
the $y_j$-th quantile of the distribution is $x_j$,
are approximated as well. In this case the goal function of the combined optimization problem is:

\begin{equation}
\begin{split}
\min_{\bm{a}, \bm{\gamma}, \bm{Z}} \quad & \sum_{i=1}^{l} w_i \left( i! (-1)^i \;\; \pi_{\alpha}(\bm{a}, \bm{\gamma}, \bm{Z}) \;\; \pi_{T}(\bm{a}, \bm{\gamma}, \bm{Z})^{-i} \;\; \mathds{1}_n \;\;- \;\; m_i \right)^2 \\
& + Q \sum_{j=1}^{J} \left( \left(1 - \pi_{\alpha}(\bm{a}, \bm{\gamma}, \bm{Z}) e^{\pi_{T}(\bm{a}, \bm{\gamma}, \bm{Z}) x_j} \mathds{1}_n \right) \;\;- \;\; y_j\right)^2
\end{split}
\label{eq:opt-unconstrained-shape}
\end{equation}

\noindent where $Q$ is a trade-off parameter for the two parts of the objective. We note that it may also be desirable to add individual weights to the points on the CDF. The rational for having individual weights for the moments in our experiments is that these are often of very different scale, and using $w_i = m_i^{-2}$ allows to treat them equally with respect to penalizing relative square deviation. the CDF values on the other hand are all in the limited range of $[0, 1]$. With this formulation, nothing more needs to be done, except for the same gradient descent procedure on the unconstrained parameters. 

For illustration, we show an example of fitting the first five moments and several quantiles of a distribution. We apply the method while varying the number of quantiles added to the objective. We expect that the higher the number of quantiles used, the closer the fitted distribution will be to the target in shape. We apply the moment-shape fitting under four different settings. The first is moment fitting only. Next, we add 3, 5, or 20 different quantiles along the CDF. When fitting 3 and 5 percentiles we used the $\{30,50,70\}$ and $\{10,30,50,70, 90\}$ percentiles, respectively. For 20 percentiles, in order to increase the shape similarity, we put more weight on small percentiles, as this appears to be more effective. In particular, we took the first 16 fitted percentiles from the range $[1, 25]$ 
uniformly, and the additional 4 percentiles are $\{30,40,50,60\}$. The value of $Q$ for this example is 0.05.

For each attempt to fit the PH distribution we plot it against the original PH, and also compute the  KL-divergence between the two. Figure~\ref{fig:shape_fits} shows the original PH in the solid blue line. It can be seen that fitting only the moments (the red dashed line '0 points') results in a very different distribution, albeit with very similar first $5$ moments. The PH fitted with 20 CDF percentiles ('20 points') is clearly the most similar to the original. The KL-divergence values are 0.099, 0.057, 0.056, and  0.004, when fitting with 0,3,5, and 20 percentiles, respectively. The KL-divergence is based on the distribution (PDF) and the moments have only an indirect effect on the KL-divergence, which makes these KL-divergence results reasonable. In all cases, the first 5 moments are fitted successfully such that the maximum relative error from any of the $5$ moments is less than $1\%$. However, fitting CDF values along moments does decrease the moment fitting accuracy, where the maximum moment error when they are jointly fitted with 0, 3, 5, and 20 CDF points is 0.020\%, 0.188\%, 0.330\%, 0.616\%, respectively. 



This demonstration indicates that the proposed method holds strong potential for accurately fitting the shape of the distribution as well as target moments. To achieve effective shape matching, several key aspects merit further investigation in future research. For instance, one should study how many percentiles to match and how to determine their specific values. Additionally, it may be worthwhile to explore fitting over the PDF instead of the CDF, or even employing a hybrid optimization strategy that simultaneously fits the CDF, PDF, and moments. Comparing our approach for shape fitting of PH distributions with existing methods would also be of interest. Given a sample trace, one can apply the EM algorithm~\cite{1467845}; while this method is known for its high accuracy, it becomes computationally expensive for large PH distributions\footnote{The runtime is also determined by other factors, such as the number of samples or the PH structure. Runtimes can be reduced when applying the EM algorithm over special PH structures such as Hyper-Erlang~\cite{1673383}}.  

In contrast, in our example, we could fit a PH distribution of size 80 in under an hour—an instance that would be impractical for the EM algorithm.



\begin{figure}[htbp]
  \centering
  \includegraphics[width=0.6\textwidth]{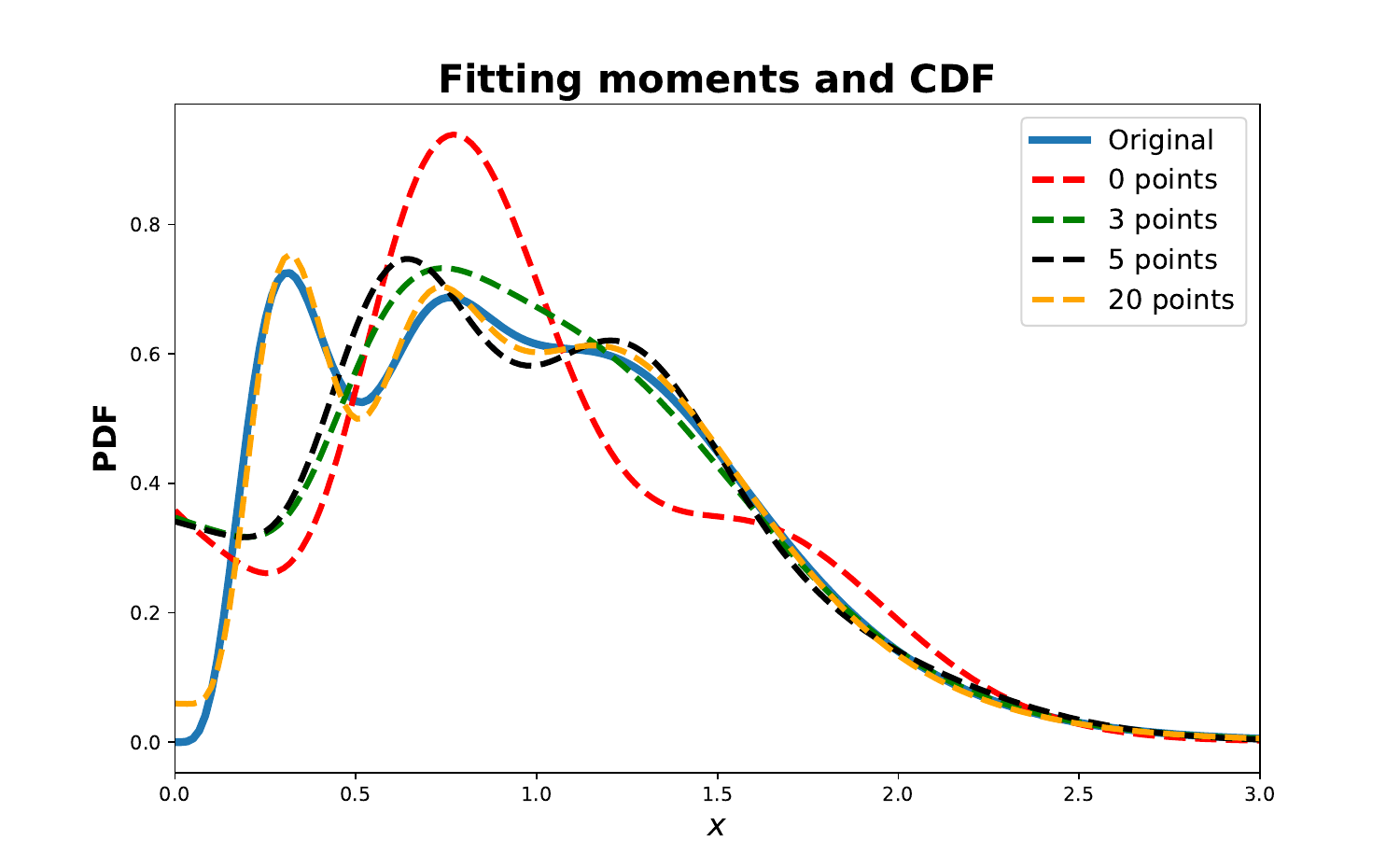}
  \caption{Shape fitting.}
  \label{fig:shape_fits}
\end{figure}

\section{A Queueing Application}\label{sec:queue}

This section demonstrates the importance of moment fitting in queueing applications. For this purpose, we analyze the stationary number of customers in the system in a $GI/GI/1$ queue. We analyze a case study where we approximate the moments of inter-arrival and service time, but not their distributions. We fit each one with its corresponding PH distribution, which is then used to analyze $PH/PH/1$ approximation of the  $GI/GI/1$ queue. 

In this example, we derive the stationary distribution when approximating the first 2, 3, 4 and 5 moments of inter-arrival and service time, and then compare it to the ground truth. To compute the stationary distribution of the $PH/PH/1$, we use the Quasi-Birth and Death (QBD) method. QBD allows us to compute the stationary number of customers in the system in a single station queue where both the inter-arrival and service time distributions follow a PH distribution (see section 21.4 in~\cite{Harchol2013}).  

This evaluation is conducted in a realistic setting, where only a stream of timestamped observations—such as arrivals, service commencements, and service completions—is available, reflecting the type of event-log data commonly encountered in real-world scenarios. Given the data, one can estimate the moments using various statistical methods. 

This example assumes the true 2 to 5 first inter-arrival and service time moments are known. Of course, in real-life scenarios, such information can only be estimates, which will induce accuracy loss. Yet, we use the true moments here in order to show the connection between the number of moments used and the stationary distribution accuracy under 'lab-conditions'.

\begin{figure}[ht]
    \centering
    \begin{subfigure}[b]{0.48\textwidth}
        \centering
        \includegraphics[width=\textwidth]{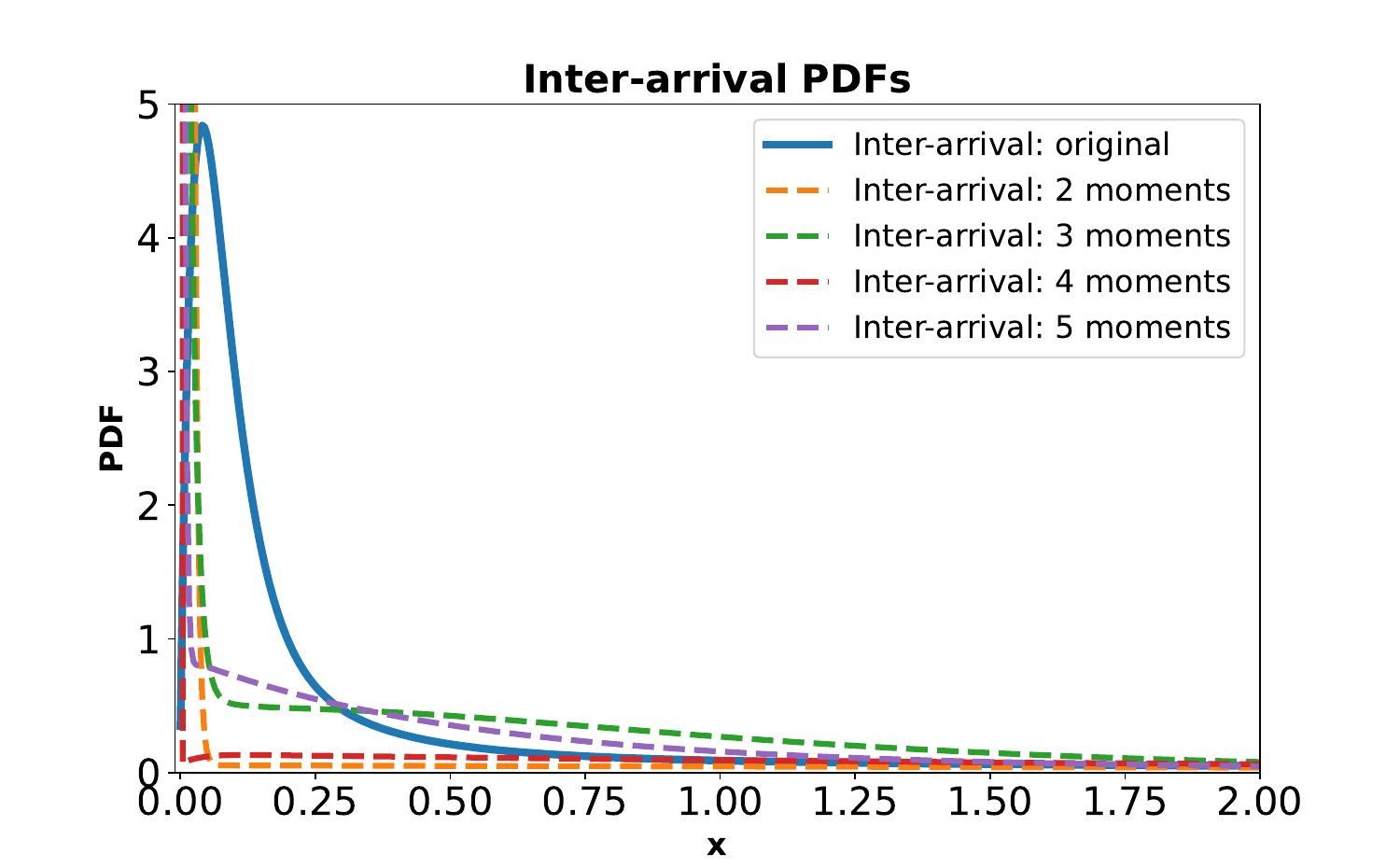}
        \caption{Inter-arrival  distribution.  }
        \label{fig:diff_arrive_times}
    \end{subfigure}
    \hfill
    \begin{subfigure}[b]{0.48\textwidth}
        \centering
        \includegraphics[width=\textwidth]{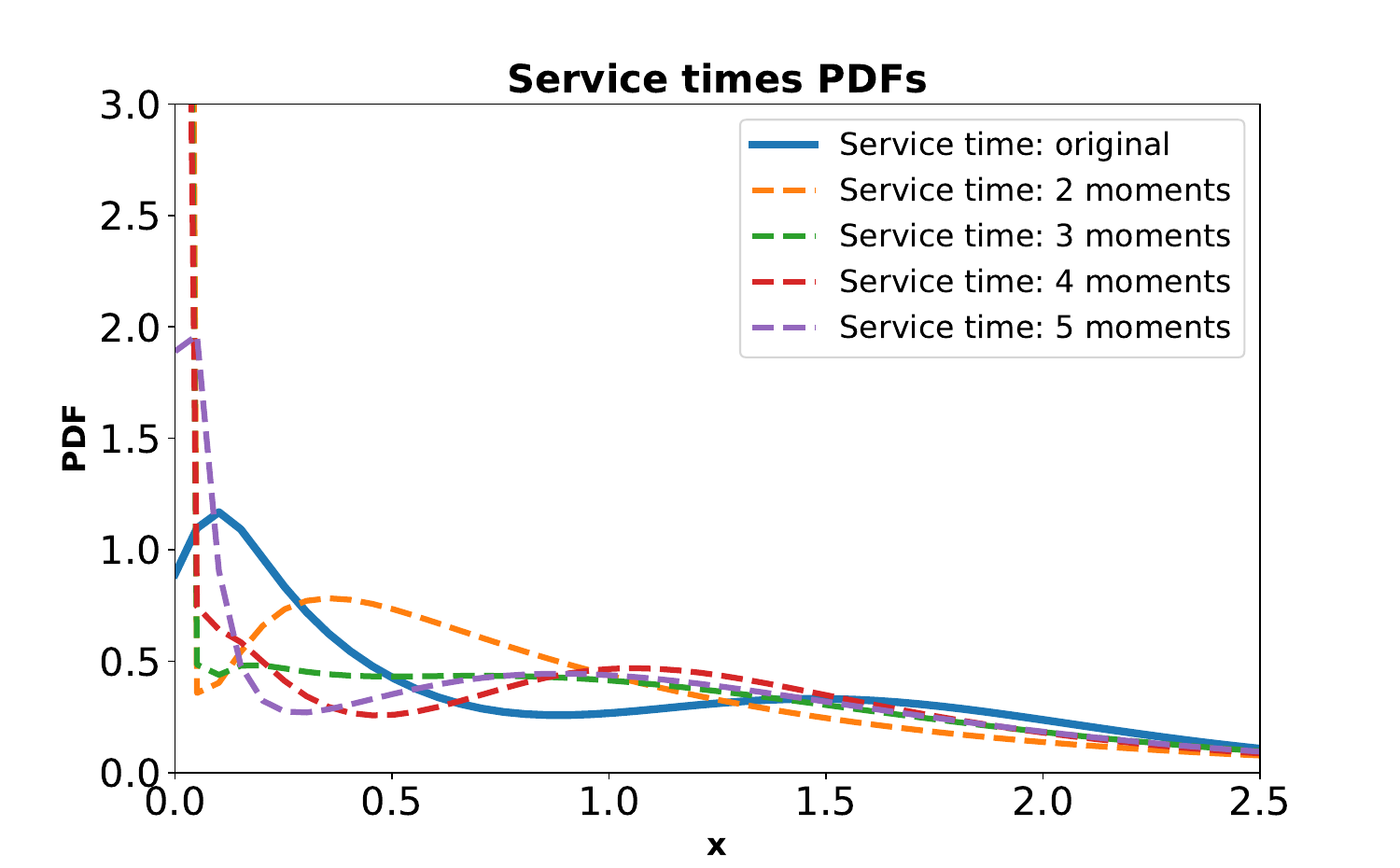}
        \caption{Service time distributions}
        \label{fig:diff_ser_times}
    \end{subfigure}
    \caption{$GI/GI/1$ analysis.}
    \label{fig:queueing}
\end{figure}

The inter-arrival and service distributions are presented in Figures~\ref{fig:diff_arrive_times} and~\ref{fig:diff_ser_times}, respectively. We present not only the original distributions (marked in a solid blue line), but also the fitted distributions based on the first 2 to 5 moments (marked in the dashed lines). As Figures~\ref{fig:diff_arrive_times} and~\ref{fig:diff_ser_times} demonstrate, it appeared that these are all completely different $GI/GI/1$ systems. Furthermore, in these settings, the utilization level is set to be 0.7.  
The Figure depicts the true distribution (marked as 'Original' in the figure) against estimated distributions, where the inter-arrival and service time distributions are fitted via the first 2, 3, 4, and 5 moments (marked as '2 fitted moments', '3 fitted moments', '4 fitted moments', and '5 fitted moments' respectively). The results shown in this particular example suggest a decreasing marginal effect as we use more moments. Such a phenomenon is not surprising~\cite{sherzer23}. However, our method allows us to quantify the moment's impact on the queue statistics measures. 

In Figure~\ref{fig:gg1}, we present the absolute error of the PMF of the number of customers in the system (up to 10 customers for visualization purposes), where ``$i$ fitted moments'' mean that both the inter-arrival and the service time are fitted based on the first $i$ moments. 

We propose the following error metric to examine the accumulated error. Let $p_i$ and $\hat{p}_{i,l}$ be the actual and estimated probability of having $i$ customers in the system, respectively, where $l$ represents the number of moments used. We wish to compute the accumulated absolute errors for each value of the number of customers, where the value at $\infty$ is equivalent to the Wasserstein-1 distance. 
Formally, 

\begin{align*}
    Accumulated \thinspace Error(j, l) = \sum_{i=0}^j |p_i - \hat{p}_{i,l}|
    \end{align*}

This measure allows us to evaluate the cumulative distance between the true distribution and the one computed via the fitted inter-arrival and service time distributions.  In Figure~\ref{fig:gg1_errors} we present the $Accumulated \thinspace Error(j, l)$ for $j \geq 0$,  for $l \in \{2, 3, 4,5\}$. As the graph depicts, incorporating more moments generally improves accuracy, though the marginal benefit diminishes with higher-order moments. Notably, the error becomes nearly negligible when the first five moments are utilized.

This compelling case study opens the door to promising future research directions. Leveraging our method, one can conduct an in-depth analysis of stochastic processes to examine the influence of higher-order moments—specifically, the $i^{\text{th}}$ moment—of the underlying distributions. While our demonstration focuses on a $GI/GI/1$ queue, the approach is readily extendable to other queueing systems and broader classes of stochastic processes, such as inventory management models. The strength of fitting PH distributions lies not only in their flexibility—enabling approximation of a wide range of distributions—but also in their Markovian structure. This allows researchers to apply existing analytical results (e.g., for the $PH/PH/1$ queue) or to efficiently simulate the system, facilitating the analysis of virtually any stochastic model.   


\begin{figure}[htbp]
\begin{minipage}{0.48\textwidth}
\includegraphics[width=\textwidth]{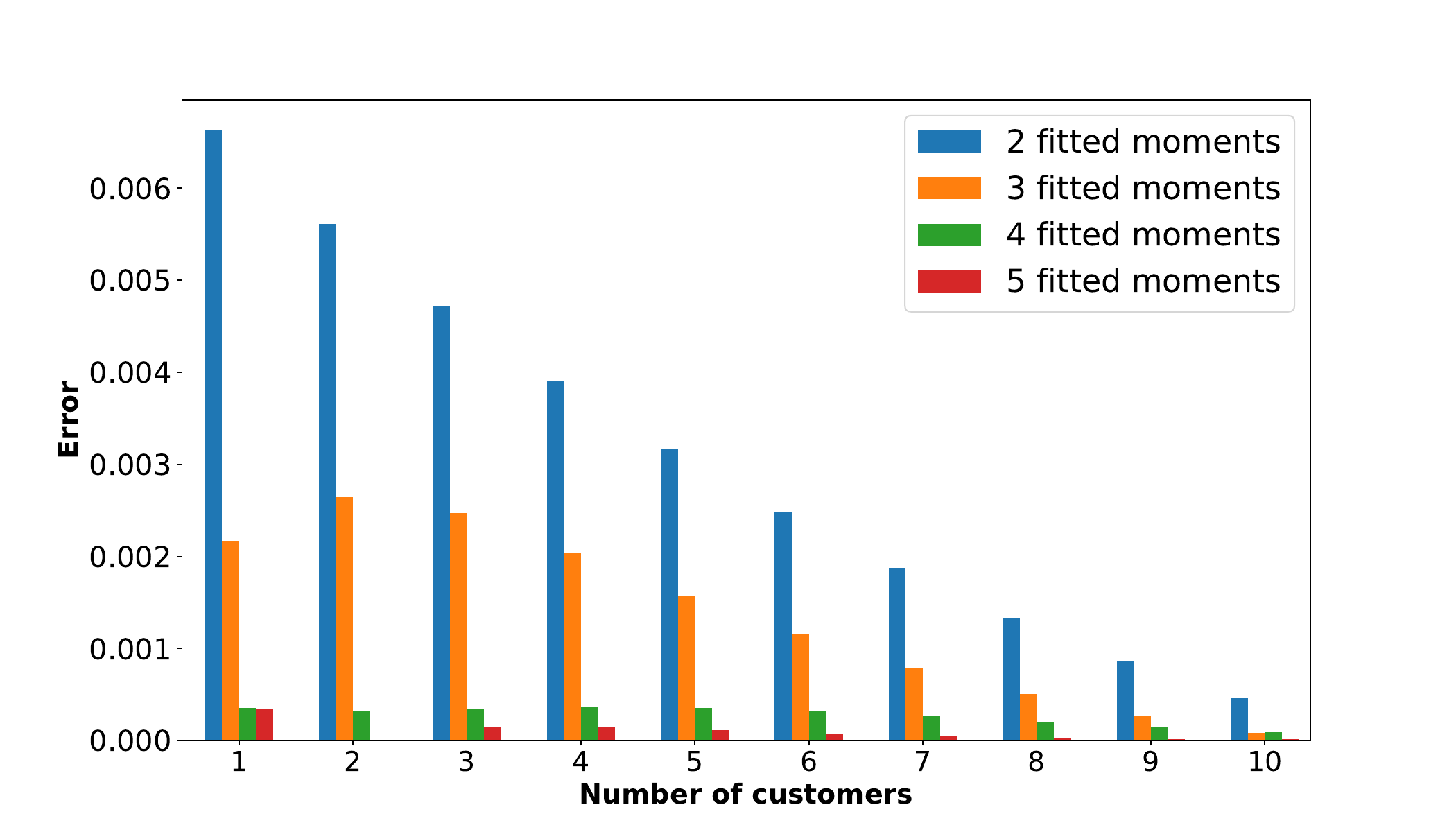}
  \caption{Absolute errors of steady-state probabilities. }
  \label{fig:gg1}
\end{minipage}
\hfill
\begin{minipage}{0.48\textwidth}
\includegraphics[width=\textwidth]{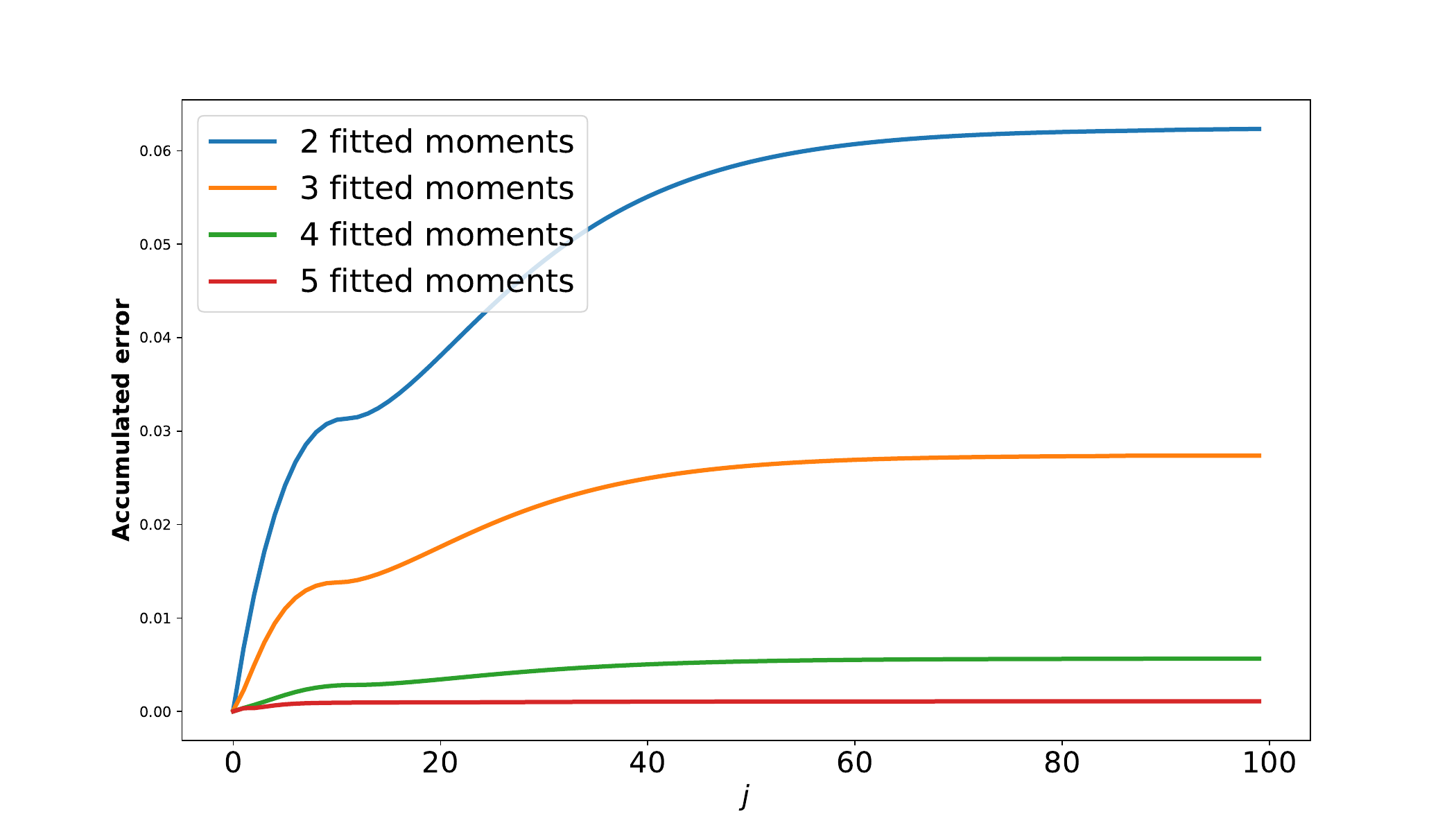}
\caption{Accumulated error of moments fitting approximation as a function of $m$ and $j$.
}
  \label{fig:gg1_errors}  
\end{minipage}
\end{figure}

\section{Conclusions}\label{sec:Conclusions}

The problem of finding a distribution that matches a set of prescribed moments arises in many modeling and simulation situations. PH distributions or subclasses thereof, such as Coxian and Hyper-Erlang, are often used as the search space, due to their generality and favorable analytical properties. While previous studies have suggested many methods appropriate for matching or fitting a small number of moments and for small PHs or special sub-classes of PH, in this paper we study the problem of fitting potentially large PHs (demonstrated up to $100$), while approximating as many as $20$ moments. 

To this end, we formulate the weighted regression optimization problem of moment fitting. By the basic properties of PH distributions this is a constrained matrix-polynomial optimization problem that does not have any direct scalable solution approach we are aware of. In order to make the problem tractable we propose a re-parametrization of the PH distribution in an unconstrained space. The mapping between the two representations is differentiable, allowing a gradient descent optimization procedure on the weighted moment fitting objective in the new unconstrained representation, and guaranteeing that the result is a valid PH. Likewise, we propose two additional re-parametrization similar in nature, for the Coxian and Hyper-Erlang subsets of the general PH.

Results of numeric experiments on a diverse dataset of moments show the method is generally able to fit sequences of up to $20$ moments to within $<1\%$ error from the worst one, using PH distributions of up to size $100$. To the best of our knowledge this is the first method to obtain such results. 

Overall, as expected, success rates are decreasing with number of required moments and increasing with PH size. Run time, on the other hand, reflects a trade-off where for larger PHs optimization is easier as there are presumably more parameter combinations that are appropriate for a given target moment sequence, but on the other hand each atomic computation is slower due to the size of the subgenerator matrix.

A limitation of the current approach is that it does not obtain a minimal size PH. This property is sometimes considered beneficial, and future work may include adaptations of the current method that to obtain minimal size. Other promising directions for future work include the application of the re-parametrization and weighted regression framework to distribution (shape) fitting, possibly in conjunction with moment fitting.

\bibliographystyle{ACM-Reference-Format}


\clearpage
\appendix
\section{
Evaluation of previous methods}
\label{sec:appA}

In this section, we evaluate the optimization method proposed in ~\cite{buchholz2009heuristic}\footnote{The authors thank the help of Peter Buchholz for providing the source code of the procedure.} as a function of a number of moments and PH size (see also a brief description of the method in Section \ref{sec:liter}). In order to evaluate the method we sampled $100$ examples from the Coxian test dataset used in the evaluation of our method (see Section \ref{sec:experiment}). Moment fitting for each example was attempted for each combination of the number of moments $l \in \{5, 10, 20\}$, and PH size $ n \in \{10, 20, 50 \} $. The max number of iterations was set to $100,000$, and all other parameters were left at the default value. For each example and combination of number of moments and PH size, the run is considered a success if all moments of the resulting fitted PH distribution fall within $1\%$ of the targets. In Table~\ref{tab:evaluate-appendix} we report the percent success for each combination, which shows much lower success rate than the related results in Figure \ref{fig:cox1}. 



\begin{table}[hb]
\centering
\begin{tabular}{l|lll}
\toprule
l \textbackslash{} n & 10   & 20   & 50  \\ \hline
5                  & 43\% & 28\% & 0\% \\
10                 & 27\% & 0\%  & 0\% \\
20                 & 0\%  & 0\%  & 0\% \\ 
\bottomrule
\end{tabular}
\caption{Percent of a random 100 examples from the Coxian test resulting in a fit within $1\%$ of each moment, for each combination of number of moments (rows) and PH size (columns).}
\label{tab:evaluate-appendix}
\end{table}

\section{
Model settings}
\label{sec:appB}

In this appendix we complete the details not provided in the main text of the implementation used to obtain the results reported in this paper. The experimental setting of the parameters presented in this section are resulted from the accuracy -- computational complexity trade off learned from to our numerical experiments. A more established optimization of these parameters is the subject of future research. 

\paragraph{Hyper-Erlang block sizes:} 
For the Hyper-Erlang re-parametrization experiments, for each value of the overall size $n \in \{20, 50, 100\}$ there is an additional choice of the block sizes $\{d_1, ...,d_k\}$ (the sum of which must be exactly $n$). In all experiments reported in this paper we use the following values. For $n=20$, we use $\{3, 4, 6, 7\}$. For $n=50$, $\{3, 4, 6, 7, 8, 10, 12\}$. Lastly, for $n=100$, $\{3, 4, 6, 7,8, 10, 10, 10,10, 12, 20\}$. 

\paragraph{Early termination of processes:} In the gradient descent optimization procedure we begin in all cases from $10,000$ starting points. However, in order to reduce the computation burden we terminate most of the processes early, by keeping only the best ones according to their objective function values. This is done according to the schedule reported in Table~\ref{tab:scheduler}. The benefit of this approach is that by sampling many initial points we are able to increase the likelihood of starting from points with a good convergence trajectory for the target moments. In this paper we do not aim to optimize this schedule but future work should examine the tradeoffs between the computation burden of many initial processes and the benefit from the multiple trials.

\begin{table}[!htp]
\centering
\caption{The schedule for number of processes kept running at each step of the Gradient Descent multi-process optimization procedure.}\label{tab:scheduler}
\begin{tabular}{l|l}
\toprule
Step & \# processes kept \\ 
\hline 
1 & 10000 \\
500 &2000 \\
5000 &200 \\
15000 &20 \\
\bottomrule
\end{tabular}
\end{table}

\section{
Notation
}
\label{sec:appC}

\begin{table}[hb]
\centering
\begin{tabular}{ll}
\toprule
notation & meaning \\
\hline
$\bm{\alpha}$  &  initial probability vector for general a PH distribution. \\ 
$\bm{T}$       &  subgenerator matrix for a general PH distribution \\
$n$            &  number of phases in the PH, that is $\bm{T}  \in \mathbb{R}^{n \times n}$. \\ 
$l$            & number of moments to fit \\ 
$m_i$          & The $i$-th moment of a PH distribution \\
$\mathds{1}_n$ & the length $n$ vector of ones \\ 
$\bm{I}_n$     & order $n$ identity matrix \\ 
$(\bm{a}, \bm{\gamma}, \bm{Z})$ & parameters of the general PH reparameterization \\ 
\hline
$(\bm{\lambda}, \bm{p})$ &  parameters for the Coxian distribution \\
$(\bm{\gamma}, \bm{u})$ & parameters of the Coxian reparameterization. Notice $\bm{\gamma}$ here plays the same role as in the general  reparameterization above. \\ 
\hline 
$(\bm{\omega}, \bm{\lambda})$ & parameters for the Hyper-Erlang distribution. Notice $\bm{\lambda}$ plays the same role as in the Coxian case \\ 
$k$  & number of blocks in a Hyper-Erlang PH. \\ 
$(\bm{\beta}, \bm{\delta})$ & parameters of the Hyper-Erlang reparameterization. \\

\bottomrule
\end{tabular}
\caption{A concise summary of the main notation used in this paper.}
\label{tab:notation}
\end{table}

\end{document}